\pdfoutput=1

\newif\ifwias
\wiasfalse

\newif\ifdraftmode
\draftmodefalse

\ifdraftmode
    \documentclass[a4paper,12pt,twoside,reqno,tbtags]{amsart}
    \usepackage{layouts}
    \newcommand{\revision}[1]{{\color{MidnightBlue}#1}}
    \newcommand{\appr}[2]{{\color{ForestGreen}\ensuremath{#1_{#2}}}}  %
    \newcommand{\eppr}[2]{{\color{ForestGreen}\ensuremath{#1_{#2,n}}}}  %
\else
    \documentclass[a4paper,12pt,twoside,reqno,tbtags]{amsart}
    \newcommand{\revision}[1]{#1}
    \newcommand{\appr}[2]{\ensuremath{#1_{#2}}}  %
    \newcommand{\eppr}[2]{\ensuremath{#1_{#2,n}}}  %
\fi

\usepackage[utf8]{inputenc}             %
\usepackage[T1]{fontenc}                %
\usepackage{lmodern}                    %
\usepackage{bbm}                        %

\usepackage[english]{babel}             %
\usepackage{csquotes}                   %

\ifwias
    \usepackage{wiaspreprint}
\else
    \usepackage[automark, draft=false]{scrlayer-scrpage}  %
\fi

\ifdraftmode
    \usepackage[protrusion=true, expansion=true, final=true]{microtype} %
\else
    \usepackage[protrusion=true, expansion=true]{microtype} %
\fi
\usepackage[format=plain, labelfont=bf]{caption} %
\usepackage{subcaption}
\usepackage{ulem}                       %
\usepackage{amsmath}                    %
\usepackage{amssymb}                    %
\usepackage{amsfonts}                   %
\usepackage{amsthm}                     %
\usepackage{amscd}                      %
\usepackage{mathtools}                  %
\mathtoolsset{showonlyrefs=true}        %
\usepackage{latexsym}                   %

\usepackage{cancel}                     %
\usepackage{esint}                      %
\usepackage{braket}                     %
\usepackage{mathrsfs}                   %
\usepackage{sansmath}
\renewcommand{\vec}[1]{\mathbf{#1}}     %

\usepackage{booktabs}                   %
\usepackage{multirow}                   %
\usepackage[dvipsnames]{xcolor}         %
\usepackage{colortbl}                   %
\usepackage{paralist}                   %
\setdefaultenum{(I)}{(i)}{}{}           %

\usepackage{graphicx}                   %
\usepackage[rightcaption]{sidecap}      %
\usepackage{wrapfig}
\usepackage{epstopdf}                   %
\usepackage{tikz}                       %
\usepackage[skins]{tcolorbox}

\usepackage{pdfpages}                   %
\usepackage{placeins}

\ifdraftmode
    \usepackage[english,colorinlistoftodos,textsize=tiny]{todonotes} %
    \makeatletter
    \renewcommand{\todo}[2][]{\@bsphack\@todo[#1]{\textcolor{black}{#2}}\@esphack\ignorespaces}
    \makeatother
    \usepackage[notref,notcite]{showkeys}
    
\else
    
    \usepackage[disable]{todonotes} %
\fi

\usepackage[ruled, vlined]{algorithm2e}

\usepackage{hyperref}

\usepackage[
    backend=biber,
    natbib=true,
    style=numeric-comp,
    sorting=none,
    maxnames=2,
    giveninits=true,
    url=false,
    arxiv=pdf,
    doi=false,
    eprint=true,
    isbn=false,
    date=terse,  %
    block=ragged,
    maxbibnames=99
]{biblatex}
\newtheorem{theorem}{Theorem}[section]
\newtheorem{lemma}[theorem]{Lemma}
\newtheorem{corollary}[theorem]{Corollary}

\newtheorem{definition}[theorem]{Definition}

\newtheorem{example}[theorem]{Example}

\newtheorem{remark}[theorem]{Remark}

\usetikzlibrary{fit, shapes, arrows, patterns, backgrounds, fadings, calc, positioning, shadows, shadings}
\usetikzlibrary{decorations.shapes, decorations.pathreplacing, decorations.pathmorphing, decorations.markings}

\tikzset{>=latex}

\definecolor{dred}{rgb}{0.8,0,0}
\definecolor{dgreen}{rgb}{0,0.45,0}
\graphicspath{
	{figures/}
}

\tikzset{
    diagonal fill/.code 2 args={
        \pgfdeclareverticalshading[%
            tikz@axis@top,tikz@axis@middle,tikz@axis@bottom%
        ]{diagonalfill}{100bp}{%
            color(0bp)=(tikz@axis@bottom);
            color(50bp)=(tikz@axis@bottom);
            color(50bp)=(tikz@axis@middle);
            color(50bp)=(tikz@axis@top);
            color(100bp)=(tikz@axis@top)
        }
        
        \tikzset{shade, left color=#1, right color=#2, shading=diagonalfill}
    }
}

\tikzset{
    diagonal bar/.code 2 args={
        \pgfdeclareverticalshading[%
            tikz@axis@top,tikz@axis@middle,tikz@axis@bottom%
        ]{diagonalbar}{6bp}{%
            color(0bp)=(tikz@axis@bottom);
            color(50bp)=(tikz@axis@bottom);
            color(50bp)=(tikz@axis@middle);
            color(50bp)=(tikz@axis@top);
            color(1000bp)=(tikz@axis@top)
        }
        
        \tikzset{shade, left color=#1, right color=#2, shading=diagonalbar}
    }
}

\tikzset{Basic/.style={
	rectangle,
	inner sep=2pt,
	minimum width=2.2em,
	minimum height=2.2em,
	rounded corners,
	line width=0.3mm
}}

\tikzset{nix/.style={
	Basic,
    draw=none,
	fill=none
}}

\tikzset{T/.style={
    Basic,
	draw=black, 
	text=black,
	fill=black!10
}}

\tikzset{GT/.style={
    Basic,
	draw=gray,
	text=gray,
	fill=black!5
}}

\tikzset{DT/.style={
	T,
	diagonal bar={black!40}{black!10},
	shading angle=45
}}

\tikzset{GDT/.style={
	GT,
	diagonal bar={black!20}{black!5},
	shading angle=45
}}

\tikzset{LT/.style={ 
	T,
	diagonal fill={black!40}{black!10}, 
	shading angle=45 
}}

\tikzset{LTI/.style={ 
	T,
	diagonal fill={black!10}{black!40}, 
	shading angle=45 
}}

\tikzset{RT/.style={
	T, 
	diagonal fill={black!40}{black!10},
	shading angle=-45
}}

\tikzset{RTI/.style={
	T, 
	diagonal fill={black!10}{black!40},
	shading angle=-45
}}

\tikzset{GLT/.style={ 
	GT,
	diagonal fill={black!20}{black!5},
	shading angle=45 
}}

\tikzset{GRT/.style={ 
	GT,
	diagonal fill={black!20}{black!5},
	shading angle=-45 
}}

\newcounter{sarrow}

\newcommand{\mbb}[1]{\mathbb{#1}}
\newcommand{\mcal}[1]{\mathcal{#1}}

\DeclareMathOperator*{\argmin}{arg\,min}

\DeclareMathOperator*{\esssup}{ess\,sup}

\DeclarePairedDelimiter{\pars}{\ensuremath{(}}{\ensuremath{)}}
\DeclarePairedDelimiter{\bracs}{\ensuremath{[}}{\ensuremath{]}}
\DeclarePairedDelimiter{\braces}{\ensuremath{\{}}{\ensuremath{\}}}

\DeclarePairedDelimiter{\floor}{\lfloor}{\rfloor}
\DeclarePairedDelimiter{\inner}{\langle}{\rangle}
\DeclarePairedDelimiter{\norm}{\|}{\|}
\DeclarePairedDelimiter{\abs}{\lvert}{\rvert}
\DeclarePairedDelimiter{\absDagger}{\dagger}{\dagger}
\DeclarePairedDelimiter{\normDagger}{{\dagger\hspace{-0.24em}\dagger}}{{\dagger\hspace{-0.24em}\dagger}}

\makeatletter
\newcommand{\opnorm}{\@ifstar\@opnorms\@opnorm}
\newcommand{\@opnorms}[1]{%
  \left|\mkern-1.5mu\left|\mkern-1.5mu\left|
   #1
  \right|\mkern-1.5mu\right|\mkern-1.5mu\right|
}
\newcommand{\@opnorm}[2][]{%
  \mathopen{#1|\mkern-1.5mu#1|\mkern-1.5mu#1|}
  #2
  \mathclose{#1|\mkern-1.5mu#1|\mkern-1.5mu#1|}
}
\makeatother

\DeclareMathOperator{\spn}{span}

\newcommand*{\dd}{\ensuremath{\mathrm{d}}}
\newcommand*{\dx}[1]{\ensuremath{\,\dd{#1}}}
\newcommand*{\dmx}[2]{\ensuremath{\,\dd{#1(#2)}}}

\let\oldbullet\bullet
\newlength{\raisebulletlen}
\setbox1=\hbox{$\bullet$}\setbox2=\hbox{\tiny$\bullet$}
\renewcommand\bullet{\raisebox{\raisebulletlen}{\,\tiny$\oldbullet$}\,}

\let\oldcolon\colon
\renewcommand{\colon}{\,\oldcolon\,}

\newcommand{\iu}{{\mathrm{i}\mkern1mu}}

\mathchardef\mhyphen="2D
\usepackage{mdframed}
\renewcommand{\qedsymbol}{\ensuremath{\blacksquare}}
\usepackage{dutchcal}  %
\usepackage{stmaryrd}  %
\usepackage{xfrac}     %
\usepackage{caption}   %

\newtheorem*{main_result}{Main result}

\newcommand{\rip}[2]{\ensuremath{\operatorname{RIP}_{#1}\pars{#2}}}
\newcommand{\apprum}{\appr{u}{\mcal{M}}\xspace}
\newcommand{\epprum}{\eppr{u}{\mcal{M}}\xspace}

\title{Convergence bounds for empirical nonlinear least-squares}
\date{\today}

\author[Eigel]{Martin Eigel}
\address{Weierstrass Institute\\Mohrenstrasse 39\\D-10117 Berlin\\Germany}
\email{martin.eigel@wias-berlin.de}

\author[Schneider]{Reinhold Schneider}
\address{TU Berlin\\Stra{\ss}e des 17. Juni 136\\D-10623 Berlin\\Germany}
\email{schneidr@math.tu-berlin.de}

\author[Trunschke]{Philipp Trunschke}
\address{TU Berlin\\Stra{\ss}e des 17. Juni 136\\D-10623 Berlin\\Germany}
\email{ptrunschke@mail.tu-berlin.de}

\makeatletter
\AtBeginDocument{%
  \patchcmd{\lbx@us@mkdaterangetrunc@short}
    {\endngroup}
    {\endgroup}
    {}{}%
}
\makeatother

\begin{document}

\maketitle

\begin{abstract}
    We consider best approximation problems in a nonlinear subset $\mathcal{M}$ of a Banach space of functions $(\mathcal{V},\|\bullet\|)$.
    The norm is assumed to be a generalization of the $L^2$-norm for which only a weighted Monte Carlo estimate $\|\bullet\|_n$ can be computed.
    The objective is to obtain an approximation $v\in\mathcal{M}$ of an unknown function $u \in \mathcal{V}$ by minimizing the empirical norm $\|u-v\|_n$.
    We consider this problem for general nonlinear subsets and establish error bounds for the empirical best approximation error.
    Our results are based on a restricted isometry property (RIP) which holds in probability and is independent of the nonlinear least squares setting.
    \revision{Several model classes are examined where analytical statements can be made about the RIP and the results are compared to existing sample complexity bounds from the literature.
    We find that for well-studied model classes our general bound is weaker but exhibits many of the same properties as these specialized bounds.
    Notably, we demonstrate the advantage of an optimal sampling density (as known for linear spaces) for sets of functions with sparse representations.}
\end{abstract}

\section{Introduction, Scope, Contributions}
\label{sec:introduction}

We consider the problem of estimating an unknown function $u$ from noiseless observations.
For this problem to be well-posed, some prior information about $u$ has to be assumed,
which often takes the form of regularity assumptions.
To make this notion more precise, we assume that $u$ is an element of some Banach space of functions $\pars{\mcal{V}, \norm{\bullet}}$
that can be well approximated in a given \revision{nonlinear subset (or \emph{model class})} $\mcal{M} \subseteq\mcal{V}$.
The approximation error is measured in the norm

\begin{equation}
\label{eq:stdnorm}
    \norm{v}
    \coloneqq \pars*{\int_Y \abs{v}_y^2\dx{\rho}\pars{y}}^{1/2},
\end{equation}
where $Y$ is some Borel subset of $\mbb{R}^d$, $\rho$ is a probability measure on $Y$ and $\abs{\bullet}_y$ is a $y$-dependent seminorm for which the integral above is finite for all $v\in\mcal{V}$.
This norm is a generalization of the $L^2\pars{Y, \rho}$- and $H^1_0\pars{Y,\rho}$-norms which are induced by the seminorms $\abs{v}_y^2 = \abs{v\pars{y}}^2$ and $\abs{v}_y^2 = \norm{\nabla v\pars{y}}_2^2$, respectively.

We characterize \revision{any} best approximation \revision{\apprum in $\mcal{M}$} by
\begin{equation}
    \revision{\apprum} \in \argmin_{v\in \mcal{M}} \norm{u - v} . \label{eq:min}
\end{equation}
In general, this \revision{approximation} is not computable.
We propose to approximate \revision{\apprum} by an estimator \revision{\epprum} that is based on the weighted least-squares method which replaces the norm $\norm{v}$ by the empirical seminorm
\begin{equation}\label{eq:emp_norm}
    \norm{v}_n := \pars*{\frac{1}{n} \sum_{i=1}^n w\pars{y_i} \abs{v}_{y_i}^2}^{1/2}
\end{equation}
for a given \emph{weight function} $w$ and a sample set $\braces{y_i}_{i=1}^n\subseteq Y$ with $y_i\sim w^{-1}\rho$.
The weight function is a non-negative function $w \ge 0$ such that $\int_Y w^{-1} \dx{\rho} = 1$.
Any corresponding \emph{empirical} best approximation \revision{\epprum in $\mcal{M}$} is characterized by
\begin{equation}
    \revision{\epprum} \in \argmin_{v\in \mcal{M}} \norm{u - v}_n . \label{eq:emp_min}
\end{equation}

\revision{Given this definition we can choose $w$ such that the theoretical convergence rate of $\norm{u-\epprum} \xrightarrow{n\to\infty} \norm{u-\apprum}$ is maximized.}
Note that changing the sampling measure from $\rho$ to $w^{-1}\rho$ is a common strategy to reduce the variance in Monte Carlo methods referred to as \emph{importance sampling}.

Since $\norm{\bullet}$ is not computable in general, the best approximation error
\begin{equation}
    \norm{\revision{u - \apprum}} = \min_{v\in \mcal{M}} \norm{u - v}
\end{equation}
serves as a baseline for a numerical method founded on a finite set of samples.
We prove in this paper that the empirical best approximation error $\norm{\revision{u - \epprum}}$ is equivalent to this error with high probability.
\begin{main_result}
    \revision{For many model classes $\mcal{M}\subset\mcal{V}$ there exist positive constants $\tilde{K} = \tilde{K}\pars{u, \mcal{M}}$ such that for all $\delta\in\pars{0,1}$
    \begin{equation}
        \norm{\revision{u-\apprum}}
        \le \norm{\revision{u-\epprum}}
        \le \pars*{1 + 2\frac{\sqrt{1+\delta}}{\sqrt{1-\delta}}} \norm{\revision{u-\apprum}}
    \end{equation}
    holds with probability $1 - p$ where $\ln\pars{p} \in \mcal{O}\pars{- n \delta^2 \tilde{K}^{-2}}$.}
\end{main_result}

\revision{This result is a combination of Theorem~\ref{thm:rip} or Corollary~\ref{cor:sample_size_rip} and Theorem~\ref{thm:error_bound}.
Some classical model classes for which it holds are discussed in Section~\ref{sec:examples}.
To prove this result for general nonlinear model classes, we extend the idea of a restricted isometry property (RIP) as known from compressed sensing. %
In contrast to previous specific results for linear spaces~\cite{cohen_migliorati}, sets of sparse functions~\cite{rauhut2016weigtedl1,rauhut2010compressedSensing}, and low-rank tensors~\cite{RAUHUT2017220}, the aim of this paper is to develop first results for a more general theory.
New results for low-rank tensors and a certain class of smooth functions are obtained and it is demonstrated how the theory can guide the choice of the model class $\mcal{M}$.

\revision{Despite the generality of the derived theory we observe many of the same phenomena as more specialised theories namely, the emergence of an optimal sampling measure (cf.~\cite{cohen_migliorati}), the importance of weighted sparsity (cf.~\cite{rauhut2016weigtedl1}) and the advantage of multilevel sampling (cf.~\cite{adcock2017breaking_coherence_barrier}).}
}

\subsection{Structure}

The remainder of the paper is organized as follows.
In Section~\ref{sec:related_work} we aim to provide a brief overview of previous work \revision{and introduce the notion of the restricted isometry property (RIP)}.
Based on the RIP, Section~\ref{sec:main_results} develops the central results of this work.
These are applied to some common model classes in Section~\ref{sec:examples}.
We begin by considering linear spaces in Section~\ref{sec:linear}.
Section~\ref{sec:sparse} considers sets of sparse functions and Section~\ref{sec:cp-tensors} examines sets of low-rank functions.
\revision{Finally, we investigate the influence of the seminorm the convergence in Section~\ref{sec:Hk_seminorm}.}
We conclude in Section~\ref{sec:discussion} with a discussion of the derived results and an outlook on future work.

\subsection{Related work}\label{sec:related_work}

\revision{When $\abs{v}_y =\abs{v\pars{y}}$ is used, \epprum} is known as the nonlinear least squares estimator of $u$.
The extensive interest in machine learning in recent years has lead to the investigation of this estimator for special model classes like sparse vectors~\cite{tao_sparse_coding, eldar2012compressed, rauhut2016weigtedl1}, low-rank tensors~\cite{tao2010matrix_completion, yuan2015tensor_completion, ESTW19, RAUHUT2017220, grasedyck2019SALSA} and neural networks~\cite{grohs_PDE_ERM, kutyniok2019deepPDEs}.
However, to the knowledge of the authors no investigation for general model classes has been published so far.
\revision{
This may be due to the fact that sparse vectors and low-rank tensors were the first model classes for which rigorous theories were developed and that most of these works focus on $\ell^1$ and nuclear norm minimization.
Our work may be regarded as an extension of these works (in particular of infinite-dimensional compressed sensing~\cite{Adcock2017infiniteDimensionalCS,adcock2017breaking_coherence_barrier}) to the nonlinear least-squares setting.
For a more in-depth discussion of statistical learning theory we refer to
the articles~\cite{Vapnik_1982,cucker_smale} and the monographs~\cite{cucker_zhou_2007,Gyoerfi2002}.
For linear spaces the first estimate in Theorem~\ref{thm:error_bound} has already appeared in~\cite{cohen_migliorati} for weighted least squares and in~\cite{migliorati2014polynomial_least_squares,chkifa2015polynomial_least_squares_for_pdes,migliorati2015polynomial_least_squares_with_noise} for standard least squares.
}

\revision{A convergence bound for the nonlinear least squares approximation problem was recently analysed in~\cite{ESTW19}.
However, the probability of the bound failing increases exponentially as the best approximation error $\norm{\pars{1-P} u}$ approaches zero and becomes one when $\norm{\pars{1-P} u}$ vanishes.
Moreover, this bound only holds for model classes that are bounded in $L^\infty$ and it does not provide any insight on what property of the set influences the convergence rate.}

The empirical approximation problem~\eqref{eq:emp_min} was thoroughly examined in~\cite{cohen_migliorati} for linear model spaces.
There the model class $\mcal{M}$ is assumed to be the $m$-dimensional subspace spanned by the \emph{orthonormal} basis functions $\braces{\vec{B}_j}_{j\in\bracs{m}}$ in $\mcal{V} = L^2\pars{Y, \rho}$.
\revision{A key point in this work is that the error $\norm{u - \epprum}$ can be bounded by $\norm{u - \apprum}$ if $\norm{\vec{G} - \vec{I}_m}_2 \le \delta < 1$}
where
\revision{
\begin{align*}
    \vec{G}
    &:= \frac{1}{n}\sum_{i=1}^n w\pars{y_i} \vec{B}\pars{y_i} \vec{B}\pars{y_i}^\intercal \\
    &= \frac{1}{n}\sum_{i=1}^n w\pars{y_i} \bracs{\vec{B}_1\pars{y_i}\ \ldots\ \vec{B}_m\pars{y_i}}^\intercal \bracs{\vec{B}_1\pars{y_i}\ \ldots\ \vec{B}_m\pars{y_i}}   
\end{align*}
}
is the Monte Carlo estimate of  the Gram matrix $\vec{I}_m$.
This condition is in fact equivalent to the \revision{norm equivalence}  %
\begin{equation}\label{eq:cm_rip}
    \pars{1-\delta} \norm{u}^2 \le \norm{u}_n^2 \le \pars{1+\delta} \norm{u}^2 \qquad \text{for } u \in \revision{\mcal{M}} .  %
\end{equation}
\citet{cohen_migliorati} prove that under suitable conditions the  %
\revision{norm equivalence}~\eqref{eq:cm_rip} is satisfied with high probability.
\begin{theorem}\label{thm:cohen_migliorati}
    If $\revision{\tilde{K}} \coloneqq \esssup_{y\in Y} w\pars{y} \boldsymbol{B}\pars{y}^\intercal \boldsymbol{B}\pars{y} < \infty$ then
    \begin{align}
        \mbb{P}\bracs{\norm{\vec{G}-\vec{I}_m}_{2} > \delta} \le 2 m \exp\pars*{-\frac{c_\delta n}{\revision{\tilde K}}},
    \end{align}
    with \revision{$c_\delta \coloneqq -\delta + \pars{1+\delta}\ln\pars{1+\delta}$}.
\end{theorem}
\revision{Equation~\eqref{eq:cm_rip} can be seen as a generalized \emph{restricted isometry property}.}
The notion of a RIP was introduced in the context of compressed sensing~\cite{tao_sparse_coding}.
It expresses the well-posedness of the problem by ensuring that $\norm{\bullet}_n$ is indeed a norm \revision{and equivalent to $\norm{\bullet}$ on $\mcal{M}$. Minimizing the error w.r.t.\ $\norm{\bullet}_n$ thus minimizes the error w.r.t.\ $\norm{\bullet}$.}
In compressed sensing of sparse vectors~\cite{tao_sparse_coding, eldar2012compressed} and low-rank tensors~\cite{RAUHUT2017220} discrete analogues of~\eqref{eq:cm_rip} are employed to derive bounds for the corresponding reconstruction errors.
A recent work which generalizes the RIP from~\cite{cohen_migliorati} to sparse grid spaces is~\cite{bohn2018convergence}.

In this paper we extend the cited results to more general norms and nonlinear model sets by directly bounding the probability of
\begin{equation}\label{eq:rip}
    \rip{A}{\delta} :\Leftrightarrow \pars{1-\delta} \norm{u}^2 \le \norm{u}_n^2 \le \pars{1+\delta}\norm{u}^2 \qquad
    \forall u\in A\subseteq\mathcal V.
\end{equation}
We prove that \revision{under some conditions on $n$ and $A$} this RIP holds with high probability
\revision{and show that these conditions are satisfied for a variety of model classes.}
We then use the RIP to provide quasi-optimality guarantees for the empirical best approximation in Theorem~\ref{thm:error_bound}.

\revision{In Remark~\ref{rmk:cones_suffice} we note that
it suffices to consider conic model sets.}
Optimizing over these sets is not straightforward.
In~\cite{traonmilin2018oneRIPtoRule}, appropriate RIP constants for exact recovery of conic model sets using a suitable regularizer are derived.

\section{Main Result}
\label{sec:main_results}

To measure the rate of convergence with which $\norm{v}_n$ approaches $\norm{v}$ as $n$ tends to $\infty$, we introduce the \emph{variation constant}
\begin{equation}
\label{eq:KAbw}
    K\pars{A}
    := \sup_{u\in A} \norm{u}_{w,\infty}^2 
    \qquad\text{with}\qquad
    \norm{v}_{w,\infty}^2 := \esssup_{y\in Y} w\pars{y} \abs{v}_y^2 .
\end{equation}
This constant constitutes a uniform upper bound of $\norm{v}_n$ for all realizations of the empirical norm $\norm{\bullet}_n$ and all $v\in A$.
We usually omit the dependence on the choice of $w$, $\abs{\bullet}_y$ and $Y$.
When a distinction between different choices of these parameters is necessary we add subscripts to $K$, respectively.

The constant $K$ is a fundamental parameter in many concentration inequalities that are used to provide bounds for the rate of convergence of the \emph{quadrature error}.
\begin{definition}[Quadrature Error]
    The \emph{quadrature error} of the empirical norm $\norm{\bullet}_n^2$
    on the model set $A\subseteq \mcal{V}$ is defined by
    \begin{equation}
    \label{eq:quadrature error}
        \mcal{E}_{A} := \sup_{u\in A} \abs{\norm{u}^2 - \norm{u}_n^2}.
    \end{equation}
\end{definition}
This error is closely related to the RIP \revision{through the}
\emph{normalization operator} $U$.
\revision{This relation is developed rigorously in the subsequent lemma.}

\revision{
\begin{definition}[Normalization Operator]
    The \emph{normalization operator}
    acts on a set $A$ by
    \begin{equation}
        U\pars{A}
        := \braces*{\tfrac{u}{\norm{u}} : u\in A\!\setminus\!\braces{0}} .
    \end{equation}
\end{definition}}
\begin{lemma}[\revision{Equivalence of RIP and a bounded quadrature error}]
\label{lem:rip_iff_gen_err_bdd}
For some set $A$,
    \begin{equation}
        \rip{A}{\delta} \Leftrightarrow \mcal{E}_{U\pars{A}} \le \delta\quad\text{ for } \delta>0.
    \end{equation}
\end{lemma}
\begin{proof}
    Note that \revision{$\norm{0}_n = \norm{0}$,} $\norm{\alpha u}_n = \abs{\alpha}\norm{u}_n$ for all $u\in A$ and $\norm{u} = 1$ for all $u\in U\pars{A}$.
    Therefore,
    \begin{equation}
        \begin{array}{lrcccll}
            & \pars{1-\delta} \norm{u}^2 & \le & \norm{u}_n^2 & \le & \pars{1+\delta}\norm{u}^2 & \forall u\in A \\
            \Leftrightarrow & \pars{1-\delta} & \le & \norm*{\frac{u}{\norm{u}}}_n^2 & \le & \pars{1+\delta} & \forall u\in A\revision{\setminus\braces{0}} \\
            \Leftrightarrow & -\delta & \le & \norm{u}_n^2 - \norm{u}^2 & \le & \delta & \forall u\in U\pars{A},
        \end{array}
    \end{equation}
    which \revision{is equivalent to} $\sup_{u\in U\pars{A}} \abs{\norm{u}^2 - \norm{u}_n^2} \le \delta$.
\end{proof}

\revision{\begin{remark} \label{rmk:cones_suffice}
    By the preceding lemma
    \begin{equation}
        \rip{A}{\delta} \Leftrightarrow \rip{\operatorname{Cone}\pars{A}}{\delta}
    \end{equation}
    where $\operatorname{Cone}\pars{A} := \braces{\alpha a : a\in A, \alpha>0}$ denotes the cone generated by $A$.
    This implies that our theory also holds for unbounded sets $A$.
\end{remark}}

We introduce the notion of a covering number to provide a well-known bound for the quadrature error in the following.

\begin{definition}[Covering Number]
    The covering number $\nu_{\norm{\bullet}}(A, \varepsilon)$ of a subset $A\subseteq\mcal{V}$ is the minimal number of $\norm{\bullet}$-open balls of radius $\varepsilon$ needed to cover $A$.
\end{definition}

\begin{lemma}\label{lem:generalization_error}
    Let $A\subseteq \mcal{V}$ and $\abs{\bullet}_y$ be such that\ $K=K\pars{U\pars{A}} < \infty$.
    Then,
    \begin{equation}
        \mbb{P}\bracs{\mcal{E}_{U(A)} \ge \delta} \le 2 \nu_{\norm{\bullet}_{w,\infty}}\pars*{U\pars{A}, \tfrac{1}{8}\tfrac{\delta}{\sqrt{K}}} \exp\pars*{-\tfrac{n}{2} \pars*{\tfrac{\delta}{K}}^2}\quad\text{for } \delta > 0.
    \end{equation}
\end{lemma}
The proof of this lemma can be found in Appendix~\ref{app:lem:generalization_error}.
With the preceding preparations we can derive a central result:
\begin{theorem}\label{thm:rip}
    Let $A\subseteq \mcal{V}$ and $\abs{\bullet}_y$ be such that\ $K=K\pars{U\pars{A}} < \infty$.
    Then,
    \begin{equation}
        \mbb{P}\bracs{\rip{A}{\delta}} \ge 1 - 2 \nu_{\norm{\bullet}_{w,\infty}}\pars*{U\pars{A}, \tfrac{1}{4}\tfrac{\delta}{\sqrt{K}}} \exp\pars*{-\tfrac{n}{2} \pars*{\tfrac{\delta}{K}}^2}\quad\text{for } \delta > 0.
    \end{equation}
\end{theorem}
\begin{proof}
    By Lemma~\ref{lem:rip_iff_gen_err_bdd} it suffices to bound the quadrature error on $U\pars{A}$.
    Lemma~\ref{lem:generalization_error} provides a bound for the probability of the \revision{complementary} event.
\end{proof}

\revision{\begin{remark}
    The variation constant $K\pars{U\pars{A}}$ can be seen as a generalization of the embedding constant $\pars{A, \norm{\bullet}} \hookrightarrow \pars{A, \norm{\bullet}_{w,\infty}}$ to nonlinear sets and therefore as an analog of $\revision{\tilde K}$ in Theorem~\ref{thm:cohen_migliorati}.
\end{remark}}

\revision{\begin{example}[$K$ is independent of the dimension]
    If $\mcal{M}$ is a manifold then one might expect the probability of \rip{\mcal{M}}{\delta} to depend on its dimension.
    But counter-examples can be constructed easily.
    Consider $\mcal{V} = L^2\pars{\bracs{-1,1}, \tfrac{\dx{x}}{2}}$ with the weight function $w\equiv 1$ and let $P_k$ denote the $k$-th Legendre polynomial.
    Let moreover $n\in\mbb{N}$ and $m \ge \tfrac{n^2}{2} + n - \tfrac{1}{2}$.
    Then the $1$-dimensional manifold $\spn\braces{P_m}$ has a larger variation constant than the $n$-dimensional manifold $\spn\braces{P_k}_{k\in\bracs{n}}$.
    We refer to Example~\ref{sec:linear} for the computation of these variation constants.
\end{example}}

\begin{corollary}[Sample Complexity]\label{cor:sample_size_rip}
    Let $\revision{c,C,}M>0$ and $A\subseteq\mcal{V}$ be a set with 
    $\nu_{\norm{\bullet}_{w,\infty}}\pars{U\pars{A}, r} \le \revision{C \pars{cr}^{-M}}$.
    Under the assumptions of Theorem~\ref{thm:rip} with $K=K\pars{U\pars{A}}$, at \revision{most
    \begin{align}
        n \le 2\pars*{M \ln\pars*{\frac{4\sqrt{K}}{c\delta}} - \ln\pars*{\frac{p}{2C}}} \pars*{\frac{K}{\delta}}^2
    \end{align}
    many samples are required to satisfy~\rip{A}{\delta} with probability $1-p$.}
\end{corollary}
\begin{proof}
    To obtain~\rip{U\pars{A}}{\delta} with a probability of $1 - p$ it suffices that
    \revision{
    \begin{align}
        p
        &\le 2 \nu_{\norm{\bullet}_{w,\infty}}\pars*{U\pars{A}, \tfrac{1}{4}\tfrac{\delta}{\sqrt{K}}} \exp\pars*{-\tfrac{n}{2}\pars*{\tfrac{\delta}{K}}^2} \\
        &\le \exp\pars*{\ln\pars{2C} - M\ln\pars*{\tfrac{c\delta}{4\sqrt{K}}} -\tfrac{n}{2}\pars*{\tfrac{\delta}{K}}^2} .
    \end{align}
    Equivalently,
    \begin{align}
        \ln\pars{p} &\le \ln\pars{2C} - M \ln\pars*{\tfrac{c\delta}{4\sqrt{K}}} - \tfrac{n}{2}\pars*{\tfrac{\delta}{K}}^2 \\
        \Leftrightarrow\qquad n &\le 2\pars{M\ln\pars*{\tfrac{4\sqrt{K}}{c\delta}} - \ln\pars{\tfrac{p}{2C}}} \pars*{\tfrac{K}{\delta}}^2 . \qedhere %
    \end{align}
    }
\end{proof}

Linear spaces, sparse vectors and low-rank tensors all satisfy the requirements of this corollary with $M$ depending linearly on the number of parameters of the model~\cite{vershynin2009sparsity, RAUHUT2017220, grohs_PDE_ERM}.
The corollary states that in these cases $n\in\mcal{O}\pars{MG}$ \revision{where the factor $G := \ln\pars{K}K^2$ represents the variation of $\norm{\bullet}_n$ on $\mcal{M}$.}
\revision{If $K$ is independent of $M$ this means that $n$ depends only linearly on $M$.}

\begin{remark}\label{rmk:K>nu}
    \revision{An interpretation of Corollary~\ref{cor:sample_size_rip} is} that the variation constant $K$ is of greater importance than the covering number $\nu$ which enters the bound on the sample complexity only logarithmically.
\end{remark}

\begin{theorem}[Empirical Projection Error]\label{thm:error_bound}
    Assume that \rip{\revision{\braces{\apprum}}-\mcal{M}}{\delta} holds.
    Then
    \begin{equation} \label{eq:error_bound_1}
        \norm{\revision{\apprum-\epprum}} \le 2\frac{1}{\sqrt{1-\delta}} \norm{\revision{u-\apprum}}_{w,\infty} .
    \end{equation}
    If in addition \rip{\braces{\revision{u-\apprum}}}{\delta} is satisfied then
    \begin{equation} \label{eq:error_bound_2}
        \norm{\revision{\apprum - \epprum}} \le 2\sqrt{\frac{1+\delta}{1-\delta}} \norm{\revision{u - \apprum}}
    \end{equation}
    \revision{and consequently
    \begin{equation} \label{eq:quasi best-approximation}
        \norm{u - \apprum} \le \norm{u - \epprum}
        \le \pars*{1 + 2 \frac{\sqrt{1+\delta}}{\sqrt{1-\delta}}}\norm{u - \apprum} .
    \end{equation}}
\end{theorem}
\begin{proof}
    First observe that $\revision{\epprum} \in \mcal{M}$ and therefore $\revision{\apprum - \epprum} \in \revision{\braces{\apprum}} - \mcal{M}$.
    \revision{By $\rip{\braces{\apprum} - \mcal{M}}{\delta}$, the triangle inequality and the definition of $\epprum$, we deduce}
    \begin{align}
        \norm{\revision{\apprum - \epprum}}
        &\le \frac{1}{\sqrt{1-\delta}}\norm{\revision{\apprum - \epprum}}_n\\
        &\le \frac{1}{\sqrt{1-\delta}}\bracs*{\norm{\revision{\apprum - u}}_n + \norm{\revision{u - \epprum}}_n}\\
        &\le 2 \frac{1}{\sqrt{1-\delta}}\norm{\revision{u-\apprum}}_n.
    \end{align}
    \revision{Hence, equation~\eqref{eq:error_bound_1} holds since $\norm{v}_n \le \norm{v}_{w,\infty}$} is satisfied for all $v\in \mcal{V}$ and in particular for \revision{$u-\apprum$}.
    \revision{Equation~\eqref{eq:error_bound_2}} follows by an application of \rip{\braces{\revision{u-\apprum}}}{\delta} \revision{and from it equation~\eqref{eq:quasi best-approximation} follows by an application of the triangle inequality to $\norm{u - \epprum}$.}
    \qedhere
\end{proof}

\begin{remark}
    This proves the main result from the introduction for $\tilde{K} = \max\braces{K\pars{U\pars{\braces{u_{\mcal{M}}}-\mcal{M}}}, K\pars{U\pars{\braces{u-u_{\mcal{M}}}}}}$.
    Note that $\ln\pars{p}\in\mcal{O}\pars{-n\delta^2\tilde{K}^{-2}}$ hides the dependence on the covering number $\nu_{\norm{\bullet}_{w,\infty}}\pars{U\pars{\mcal{M}}, r}$ which does not depend on $n$.
\end{remark}

\revision{\begin{remark}
    Theorem~\ref{thm:error_bound} bounds $\norm{\apprum - \epprum}$ even if $\apprum$ and $\epprum$ are not uniquely defined.
\end{remark}}

\begin{remark}
    \revision{Theorem~\ref{thm:error_bound} requires \rip{\braces{\apprum}-\mcal{M}}{\delta} and \rip{\braces{u-\apprum}}{\delta}.
    If the covering number of $U\pars{\braces{\apprum} - \mcal{M}}$ is finite then $K\pars{U\pars{\braces{\apprum} - \mcal{M}}}$ and $K\pars{U\pars{\braces{u - \apprum}}}$ are bounded and Theorem~\ref{thm:rip} guarantees that \rip{\braces{\apprum}-\mcal{M}}{\delta} and \rip{\braces{u-\apprum}}{\delta} hold when $n$ is large enough.
    
    If $u\in\mcal{M}$ then \rip{\braces{u - \apprum}}{\delta} is implied by \rip{\braces{\apprum}-\mcal{M}}{\delta} and bounds for the sample complexity of some well-known model classes are given in Section~\ref{sec:examples}.
    If $u\not\in\mcal{M}$ then the probability of \rip{\braces{u-\apprum}}{\delta} has to be bounded separately.
    Since $\nu\pars{U\pars{\braces{u-\apprum}}, r} = 1$, we only need to bound $K\pars{U\pars{\braces{u-\apprum}}}$ to apply Theorem~\ref{thm:rip}.
    Since $K\pars{U\pars{\braces{u - \apprum}}}$ depends only on $u-\apprum$ it is a purely approximation theoretic constant and we provide explicit bounds for two examples in the following.
    \begin{itemize}
        \item Let $\mcal{M}$ be a space of low-rank functions and $\mcal{M}_{\omega,s}$ be an appropriately defined space of sparse functions as defined in Section~\ref{sec:sparse}.
        If $u = u_{\mathrm{low\mhyphen rank}} + u_{\mathrm{sparse}}$ with $u_{\mathrm{low\mhyphen rank}}\in\mcal{M}$ and $u_{\mathrm{sparse}} \in \mcal{M}_{\omega, s}$ then $K\pars{U\pars{\braces{u - \apprum}}} \le s^2$.
        \item Consider $u\pars{x} := \sin\pars{\pi x}$ and assume that $\apprum\pars{x}= \sum_{k=0}^m (-1)^k \frac{\pars{\pi x}^{2k+1}}{\pars{2k+1}!}$.
        From this one can derive that $K\pars{U\pars{\braces{u - \apprum}}} \le 4m + 7$.
    \end{itemize}
    }

\end{remark}

\begin{remark}[Indicator for \rip{A}{\delta}] \label{rmk:indicator}
    \revision{In Theorem~\ref{thm:error_bound} there is no constraint on the samples $\braces{y_i}_{i=1}^n$ except that they satisfy the RIP.
    They explicitly do not have to be i.i.d.\ random variables.
    This means that they could theoretically be determined by a deterministic quadrature rule.
    The challenge however is to ensure the RIP.}
    In~\cite{cohen_migliorati} the empirical Gramian could be used to verify this RIP for a given sample set.
    In the nonlinear setting this is not possible.
    To obtain a practical indicator for the convergence of our method we make the following considerations.
    Define $A := \pars{\revision{\braces{\apprum}}-\mcal{M}}\cup \braces{\revision{u-\apprum}}$,  $e_n := \norm{\revision{u-\epprum}}$ and $e := \norm{\revision{u-\apprum}}$.
    Observe that for $\delta \le \frac{1}{\sqrt{2}}$
    \begin{equation}
        1 + \delta \le \sqrt{\frac{1+\delta}{1-\delta}} \le 1 + 2\delta.
    \end{equation}
    Combining the second inequality with Theorem~\ref{thm:error_bound} leads to
    \begin{align}
        \rip{A}{\delta} &\Rightarrow e_n \le \pars*{1 + 2\sqrt{\tfrac{1+\delta}{1-\delta}}} e \le \pars{1 + 2\pars{1+2\delta}} e \\
        &\Rightarrow e_n \le \pars{3 + 4\delta} e.
    \end{align}
    Therefore,
    \begin{equation} \label{eq:simple_bound}
        \mbb{P}\bracs{\rip{A}{\delta}} \le \mbb{P}\bracs{e_n \le \pars{3 + 4\delta} e} .
    \end{equation}
    By Theorem~\ref{thm:rip} there exist $c$ and $\nu\pars{\delta}$ such that
    \begin{align}
        1 - \nu\pars{\delta}\exp\pars{-cn\delta^2} &\le \mbb{P}\bracs{\rip{A}{\delta}} .
    \intertext{Combining this with~\eqref{eq:simple_bound} yields}
        1 - \nu\pars{\delta}\exp\pars{-cn\delta^2} &\le \mbb{P}\bracs{e_n \le \pars{3 + 4\delta} e} =: p\pars{\delta} . \label{eq:p_delta}
    \end{align}
    Since $p\pars{\delta}$ is increasing in $\delta$, we can define an inverse
    \revision{in the sense of the quantile function $\delta\pars{\tilde{p}} := \inf\braces{\tilde{\delta}\in\mbb{R}_{\ge 0}:\tilde{p}\le p\pars{\tilde{\delta}}}$.
    For fixed $\tilde{p} := p\pars{\delta}$ in equation~\eqref{eq:p_delta} it then follows that $\delta \ge \delta\pars{\tilde{p}} =: \tilde{\delta}$ and consequently
    \begin{equation}
        -\ln\pars{1-\tilde{p}} \ge c n \tilde{\delta}^2 - \ln\pars{\nu\pars{\tilde{\delta}}}
    \end{equation}
    or equivalently
    \begin{equation}
        -\ln\pars{1-p} \ge c n \delta\pars{p}^2 - \ln\pars{\nu\pars{\delta\pars{p}}} .
    \end{equation}}
    \revision{Since $\delta\pars{p}\ge 0$ is increasing and $-\ln\pars{\nu\pars{\delta\pars{p}}} \xrightarrow{p\to 1} 0$, the second term in the above sum becomes negligible for large $p$.
    This yields}
    \begin{equation}
        \delta\pars{p} %
        \lesssim n^{-1/2}
    \end{equation}
    from which follows that
    \begin{equation}
        e_n \le \pars{3+4n^{-1/2}} e \le 4 \pars{1 + n^{-{1}/{2}}} e .
    \end{equation}
    \revision{We can use this bound as an indicator for when \rip{A}{\delta} is attained for some $\delta\le\tfrac{1}{\sqrt{2}}$.
    To do this we select a test set of $n'$ samples and observe the test set error $\tilde{e}_n := \norm{u-\epprum}_{n'}$ as the number of samples $n$ is increased.
    When $\tilde{e}_n$ begins to decrease with a rate of $\pars{1+n^{-r}}$ we take this as an indication that \rip{A}{\delta} is satisfied and that additional sampling is unnecessary.}
    This is illustrated in Figure~\ref{fig:err_vs_nr_samples}.

    \begin{figure}[!ht]
        \scriptsize
        \includegraphics[width=\textwidth]{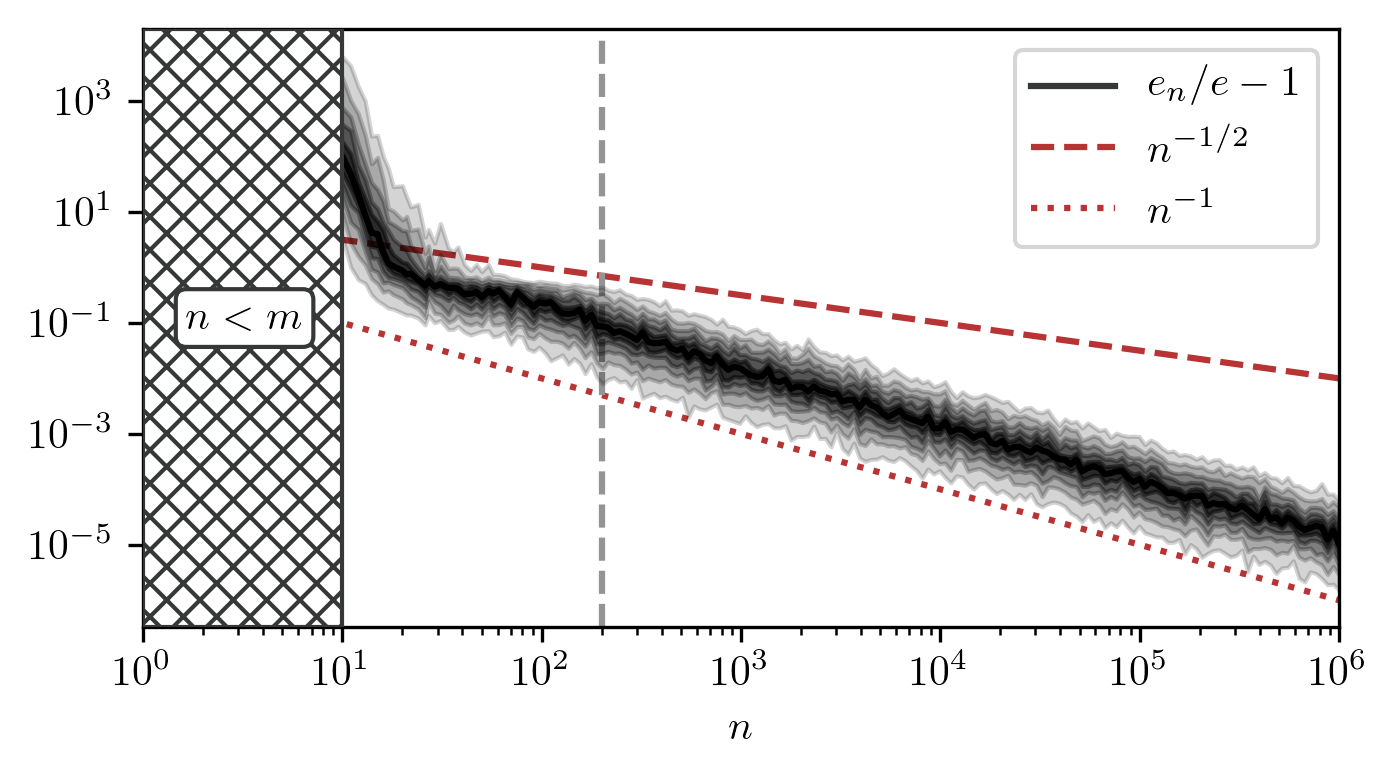}
        \caption{
            \revision{Let $\mcal{V}=L^2\pars{\bracs{-1,1}, \tfrac{\dx{x}}{2}}$, $w\equiv 1$ and $\mcal{M}$ be the model space of polynomials of degree less than $m=10$.
            Let moreover $A$, $e_n$ and $e$ be defined as in Remark~\ref{rmk:indicator}.
            Depicted is the distribution of the random variable $e_n/e-1$ for different values of $n$ and a synthetic (but fixed) function $u$.}
            The hatched area on the left marks a range of $n$ where the approximation problem is underdetermined and any error can be reached.
            When $n\ge m$ the approximation problem has a unique solution in the least squares sense. From this point until the gray and dashed line, an exponential decay of the error can be observed. This decay results from the exponentially fast convergence of the probability for \rip{A}{\delta} w.r.t.\ $n$.
            From there on, \rip{A}{\delta} holds with a high probability and the error decays with a rate of $n^{-1}$.
            \revision{Remark~\ref{rmk:indicator} predicts a rate of $n^{-1/2}$ but the condition $e_n \le c\pars{1 + n^{-r}}$ is satisfied for $c=r=1$.}
            \revision{This faster decay can be explained by the fact that for the linear space $\mcal{M}$ the bounds in the proof of Theorem~\ref{thm:rip} are suboptimal (see Example~\ref{sec:linear})}.
        }
        \label{fig:err_vs_nr_samples}
    \end{figure}
\end{remark}

\begin{remark}[Reconstruction with Noise]
    Consider the randomly perturbed seminorm $\abs{v}_y + \eta_y$ where $\eta_y$ is a centered random process satisfying the bound $w\pars{y}\eta_y^2 \le \frac{1}{4} \pars{1-\delta} \varepsilon^2$ for some $\varepsilon>0$ and $\delta \in \pars{0,1}$.
    This seminorm induces the perturbed empirical norm
    \begin{equation}
        \norm{v}_{\eta,n} := \pars*{\frac{1}{n} \sum_{i=1}^n w\pars{y_i} \pars{\abs{v}_{y_i} + \eta_{y_i}}^2}^{1/2}
    \end{equation}
    and the perturbed empirical best approximation
    \begin{equation}
        \revision{u_{\mcal{M},n,\eta}} \in \argmin_{v\in\mcal{M}} \norm{u - v}_{\eta,n} .
    \end{equation}
    Assume that \rip{\revision{\braces{\apprum}}-\mcal{M}}{\delta} holds.
    Then
    \begin{equation}
        \norm{\revision{\apprum - u_{\mcal{M},n,\eta}}} \le 2\frac{1}{\sqrt{1-\delta}} \norm{u-\apprum}_{w,\infty} + \varepsilon.
    \end{equation}
    If in addition \rip{\braces{\revision{u-\apprum}}}{\delta} is satisfied then
    \begin{equation}
        \norm{\revision{\apprum - u_{\mcal{M},n,\eta}}} \le 2\sqrt{\frac{1+\delta}{1-\delta}} \norm{u-\apprum} + \varepsilon.
    \end{equation}
\end{remark}

\revision{
\begin{remark} \label{rmk:residuals}
    A generalization of Theorem~\ref{thm:error_bound} also holds for the residual minimization problem
    \begin{equation}
        \min_{v\in\mcal{M}} \norm{u - Lv}
    \end{equation}
    since whenever \rip{\braces{u}-L\mcal{M}}{\delta} and \rip{\braces{u-Lv^*_{\mcal{M}}}}{\delta} hold we can estimate
    \begin{equation}
        \norm{u-Lv^*_{\mcal{M},n}} \lesssim \norm{u-Lv^*_{\mcal{M},n}}_n \le \norm{u-Lv^*_{\mcal{M}}}_n \lesssim \norm{u-Lv^*_{\mcal{M}}} .
    \end{equation}
    This means that the present theory can treat residual minimization problems by estimating the RIP for modified model classes.
    An important application of such a problem arises in medical imaging and is briefly discussed in Example~\ref{ex:MRI}.
\end{remark}
}

\section{Examples and numerical illustrations}
\label{sec:examples}

In this section, we examine some exemplary model spaces to which the developed theory can be applied.
More specifically, we consider linear spaces, sparse vectors and tensors of fixed rank.
The following theorem is central to the further considerations.

\begin{theorem}\label{thm:pointwise_K}
    Let $\mcal{V}$ be a separable vector space and $A\subseteq\mcal V$.
    Then the pointwise supremum $\hat b\pars{y} \coloneqq \sup_{v\in A} \abs{v}_y^2$ with respect to $y\in Y$ is measurable and \revision{for any weight function $w$}
    \begin{equation}
        K\pars{A} = \norm{w \hat{b}}_{L^\infty\pars{Y,\rho}} .
    \end{equation}
    \revision{If $A$ is $\norm{\bullet}$-bounded and $K\pars{A}$ is finite then
    \begin{equation}
        K\pars{A} \ge \norm{\hat b}_{L^1\pars{Y,\rho}},
    \end{equation}
    where the lower bound is attained by the weight function $w = \norm{\hat b}_{L^1\pars{Y,\rho}}{\hat b}^{-1}$.}
\end{theorem}
\begin{proof}
    See Appendix~\ref{app:thm:pointwise_K}.
\end{proof}

This theorem allows to analyse the seminorm and the model class independently from the choice of weight function which can be chosen optimally when these first two parameters are fixed.

\subsection{Linear Spaces}
\label{sec:linear}

\revision{
Consider an $m$-dimensional linear subspace $\mcal{V}_m \subseteq \mcal{V} := L^2\pars{Y, \rho}$ spanned by the \emph{orthonormal} basis $\braces{\vec{B}_j}_{j\in\bracs{m}}$.
Recall that Theorem~\ref{thm:pointwise_K} implies $K\pars{U\pars{\mcal{V}_m}} = \norm{w\hat{b}}_{L^\infty\pars{Y,\rho}}$ where
\begin{equation}
    \hat{b}\pars{y}
    = \sup_{\substack{v\in\mcal{V}_m\\\norm{v}=1}} \abs{v\pars{y}}^2
    = \sup_{\substack{\vec{v}\in\mbb{R}^m\\\norm{\vec{v}}_2=1}}\abs{\vec{B}\pars{y}^\intercal \vec{v}}^2
    = \norm{\vec{B}\pars{y}}_2^2 .
\end{equation}
Here, the second equality follows by orthonormality and the third by the Cauchy--Schwarz inequality.
From this, Theorem~\ref{thm:pointwise_K} implies
\begin{equation}
    K\pars{U\pars{\mcal{V}_m}} \ge \norm{\hat{b}}_{L^1\pars{Y,\rho}} = m
\end{equation}
where the optimal weight function is given by $w\pars{y} \coloneqq m\norm{\boldsymbol{B}\pars{y}}_2^{-2}$.
Note that this fact was already reported in~\cite{cohen_migliorati}.}

\revision{Using the fact that $\norm{v}_{w,\infty} \le \sqrt{K}\norm{v}$, we obtain
\begin{equation}
    \nu_{\norm{\bullet}_{w,\infty}}\pars{U\pars{\mcal{V}_m}, r}
    \le \nu_{\norm{\bullet}}\pars*{U\pars{\mcal{V}_m}, \frac{r}{\sqrt{K}}}
    \le \pars*{\frac{r}{2m}}^{-m} .
\end{equation}
Corollary~\ref{cor:sample_size_rip} then bounds the sample complexity of this model class by
\begin{equation}
    n \le 2\pars*{m\ln\pars*{\frac{8m^{3/2}}{\delta}} - \ln\pars*{\frac{p}{2}}} \pars*{\frac{m}{\delta}}^2 \in \mcal{O}\pars{m^3\ln\pars{m}}.
\end{equation}
Although our approach is more general the resulting asymptotic bound differs only by a factor of $m^2$ from the bound $n\in\mcal{O}\pars{m\ln\pars{m}}$ provided in~\cite{cohen_migliorati}.
The near optimal bound in~\cite{cohen_migliorati}
is obtained by using tighter concentration inequalities %
(cf.~\cite{user_friendly_tail_bounds}) when bounding the probability of \rip{\mcal{V}_m}{\delta} in Theorem~\ref{thm:cohen_migliorati}.}

\revision{
\begin{remark}
    When the sampling density cannot be changed, the variation constant can also be used to guide the choice of a suitable model class.
    For linear spaces this section shows that an optimal model space is spanned by an orthonormal basis for which the basis functions are bounded by $1$.
    Such spaces are characterized in~\cite{kowalski11_pointwise_bf} and a prime example is the Fourier basis of $L^2\pars{[-1,1],\frac{\dx{x}}{2}; \mbb{C}}$.
\end{remark}
}

\subsection{Sets of sparse functions}
\label{sec:sparse}

In this section we follow the ideas of~\cite{rauhut2016weigtedl1} and consider spaces with weighted sparsity constraints.
For any sequence $\omega\in \mbb{R}_{\ge0}^{\mbb{N}}$ and any subset $S\subseteq\mbb{N}$, define a weighted cardinality and a weighted $\ell^0$-seminorm by
\begin{equation}
    \omega\pars{S} := \sum_{j\in S} \omega_j^{2}
    \qquad\text{and}\qquad
    \norm{\vec{v}}_{\omega,0} := \omega\pars{\operatorname{supp}\pars{\vec{v}}} .
\end{equation}
Observe that \revision{$\omega \preceq \tilde{\omega}$ (i.e.\ $\omega_j \le \tilde{\omega}_j$ for all $j$)} implies $\omega\pars{S} \le \tilde{\omega}\pars{S}$
and that $\omega\pars{S} = \abs{S}$ for $\omega\equiv \boldsymbol{1}$.

\revision{Let in the following $\braces{\vec{B}_j}_{j\in\mbb{N}}$ be a fixed \emph{orthonormal} basis for $\mcal{V} := L^2\pars{Y,\rho}$, fix a weight function $w$ and define the model set}
\begin{equation}
    \mcal{M}_{\omega, s} := \braces{v \in \mcal{V} : \norm{\vec{v}}_{\omega,0} \le s},
\end{equation}
where $\vec{v}$ denotes the coefficient vector of $v\in\mcal{V}$ with respect to the basis $\braces{\vec{B}_j}_{j\in\mbb{N}}$.

\begin{lemma}\label{lem:weighted_sparsity_properties}
    It holds that
    \begin{itemize}
        \item \revision{$\mcal{M}_{\tilde{\omega}, s} \subseteq \mcal{M}_{\omega, s}$ for $\omega\preceq\tilde{\omega}$,}
        \item $\mcal{M}_{\omega, s} \subseteq \mcal{M}_{\omega, t}$ for $s\le t$,
        \item $\mcal{M}_{\omega,s} = -\mcal{M}_{\omega,s}$ \revision{and}
        \item $\mcal{M}_{\omega,s}+\mcal{M}_{\omega,t} \subseteq \mcal{M}_{\omega,s+t}$.
    \end{itemize}
    \revision{Moreover, if $\omega_j \ge \norm{\vec{B}_j}_{w,\infty}$ for all $j$ then $\norm{v}_{w,\infty} \le \sqrt{s}\norm{v}$ for all $v\in\mcal{M}_{\omega, s}$.}
\end{lemma}
\begin{proof}
    The first four assertions are trivial.
    To prove the last one, let $v\in\mcal{M}_{\omega,s}$.
    Using the  triangle inequality and $\omega_j \ge \norm{\vec{B}_j}_{w,\infty}$, we obtain
    \begin{align} \label{eq:Linf-bound-weighted-sparsity}
        \norm{v}_{w,\infty}
        &\le \sum_{j=1}^\infty \abs{\vec{v}_j} \norm{\vec{B}_j\pars{y}}_{w,\infty} 
        \le \sum_{j=1}^\infty \abs{\vec{v}_j} \omega_j
        = \sum_{j\in\operatorname{supp}\pars{\vec{v}}} \abs{\vec{v}_j} \omega_j.
    \intertext{The Cauchy-Schwarz inequality, $\norm{\vec{v}}_{\omega,0} \le s$ and the orthonormality of $\vec{B}$ yield}
        \revision{\norm{v}_{w,\infty}}
        &\le \norm{\vec{v}}_2 \sqrt{\sum_{j\in\operatorname{supp}\pars{\vec{v}}} \omega_j^2}
        = \norm{\vec{v}}_2 \sqrt{\norm{\vec{v}}_{\omega,0}}
        \le \norm{v} \sqrt{s} .
    \end{align}
\end{proof}

\begin{lemma} \label{lem:weighted_sparsity_variation}
    \revision{Let $\omega_j \ge \norm{\vec{B}_j}_{w,\infty}$ for all $j$. Then} $K\pars{U\pars{\mcal{M}_{\omega,s}}} \le s$.
\end{lemma}
\begin{proof}
    This follows directly from Lemma~\ref{lem:weighted_sparsity_properties}.
\end{proof}

\begin{remark}
    \revision{This setting also incorporates the standard sparsity class 
    \begin{equation}
        \mcal{M}_{\boldsymbol{1}, k} = \mcal{M}_{\boldsymbol{1}\omega_{\mathrm{max}}, k\omega_{\mathrm{max}}^2},
    \end{equation}
    where $\boldsymbol{1}=\pars{1,1\ldots}$ and $\omega_{\mathrm{max}} := \max_{j\in\bracs{m}} \norm{\vec{B}_j}_{w,\infty}$.
    This means that $K\pars{U\pars{\mcal{M}_{\boldsymbol{1},k}}} \le k\omega_{\mathrm{max}}^2$.
    When the chosen basis is a tensor product basis $\vec{B}_j = B_{j_1}\otimes\cdots\otimes B_{j_M}$ and the weight function has a product structure $w = w_1\otimes\cdots\otimes w_M$, this implies that $K\pars{U\pars{\mcal{M}_{\boldsymbol{1},k}}}$
    grows exponentially with the order $M$.
    This is a limitation when using classical isotropic sparsity for high-dimensional problems.}
\end{remark}

\begin{lemma} \label{lem:weighted_sparsity_covering}
    \revision{Let $\omega_j \ge \norm{\vec{B}_j}_{w,\infty}$ for all $j$ and let} $\mcal{V}_m$ be an $m$-dimensional subspace spanned by a subset of $\braces{\vec{B}_j}_{j\in\mbb{N}}$. Then \revision{there exists $C>0$ such that}
    \begin{equation}
        \nu_{\norm{\bullet}_{w,\infty}}\pars{U\pars{\mcal{M}_{\omega,s}\cap \mcal{V}_m}, r}
        \revision{\le} \pars*{\tfrac{\revision{C}m}{r\sqrt{s}}}^{s} .
    \end{equation}
\end{lemma}
\begin{proof}
    We show that
    \begin{equation}
        \nu_{\norm{\bullet}_{w,\infty}}\pars{U\pars{\mcal{M}_{\omega,s}\cap \mcal{V}_m}, r}
        \le \nu_{\norm{\bullet}}\pars*{U\pars{\mcal{M}_{\omega,s}\cap \mcal{V}_m}, \frac{r}{\sqrt{2s}}} 
        \revision{\le} \pars*{\tfrac{\revision{C}m}{r\sqrt{s}}}^{s}.
    \end{equation}
    For the first step, let $\braces{v_j}$ be the centers of a $\norm{\bullet}$-covering of $U\pars{\mcal{M}_{\omega,s}\cap \mcal{V}_m}$ with radius $\frac{r}{\sqrt{2s}}$.
    Thus, for any $v \in U\pars{\mcal{M}_{\omega,s}\cap \mcal{V}_m}$ there exists $v_j$ such that $\norm{v-v_j} \le \frac{r}{\sqrt{2s}}$. Since $v-v_j\in\mcal{M}_{\omega,2s}$ and by Lemma~\ref{lem:weighted_sparsity_properties},
    \begin{equation}
        \norm{v - v_j}_{w,\infty} \le \sqrt{2s} \revision{\norm{v - v_j}} \le r .
    \end{equation}
    This implies that $\braces{v_j}$ are also the centers of an $\norm{\bullet}_{w,\infty}$-covering with radius $r$.
    
    For the second step, observe that $\mcal{M}_{\omega, s} \subseteq \mcal{M}_{\boldsymbol{1}, s} = \mcal{M}_{\boldsymbol{1}, \floor{s}}$.
    Since $\pars{\mcal{V}_m, \norm{\bullet}} \simeq \pars{\mbb{R}^m, \norm{\bullet}_2}$ it remains to compute the covering number for the unit sphere of $\floor{s}$-sparse vectors in $\mbb{R}^m$. A bound for this is given in~\cite{vershynin2009sparsity} by
    \begin{equation}
        \nu_{\norm{\bullet}_2}\pars*{S^{\mbb{R}^m}_1\pars{0}\cap \mcal{M}_{\boldsymbol{1}, \floor{s}}, \frac{r}{\sqrt{2s}}}
        \revision{\le} \pars*{\frac{\revision{C}m\sqrt{s}}{r\floor{s}}}^{\floor{s}}
        \le \pars*{\frac{\revision{2C}m}{r\sqrt{s}}}^{s} .
    \end{equation}
\end{proof}

\begin{theorem}\label{thm:sparse_rip}
    \revision{Let $\omega_j \ge \norm{\vec{B}_j}_{w,\infty}$ for all $j$ and let} $\mcal{V}_m$ be an $m$-dimensional subspace spanned by a subset of $\braces{\vec{B}_j}_{j\in\mbb{N}}$. Then,
    \begin{equation}
        \mbb{P}\bracs{\neg\rip{\mcal{M}_{\omega,s}\cap\mcal{V}_m}{\delta}} \lesssim \exp\pars*{ s\ln\pars*{\tfrac{m}{\delta}} - \tfrac{n}{2} \pars*{\tfrac{\delta}{s}}^2}.
    \end{equation}
\end{theorem}
\begin{proof}
    The assertion follows directly from Theorem~\ref{thm:rip} together with Lemmas~\ref{lem:weighted_sparsity_variation} and~\ref{lem:weighted_sparsity_covering}.
\end{proof}

\begin{remark}
    Theorem~\ref{thm:sparse_rip} states a sample complexity of
    \begin{align}
        n &\lesssim s^2 \pars{s\ln\pars{m} - s\ln\pars{\delta} - \ln\pars{1-p}}\delta^{-2} .
    \intertext{%
    \revision{This} result can be compared with Theorem 5.2 in \cite{rauhut2016weigtedl1} where}
        n &\lesssim s\max\braces{\ln^3\pars{s} \ln\pars{m}, \ln\pars{p^{-1}}} \delta^{-2}
    \intertext{or Theorems~4.4 and~8.4 in~\cite{rauhut2010compressedSensing} where}
        n &\lesssim s\max\braces{\ln^2\pars{s} \ln\pars{m} \ln\pars{n}, \ln\pars{p^{-1}}} \delta^{-2} \revision{\omega_{\mathrm{max}}^2}
    \end{align}
    \revision{with $\omega_{\mathrm{max}} := \max_{j\in\bracs{m}} \norm{\vec{B}_j}_{w,\infty}$.
    Since our theory is very general, we cannot expect our bound to be as strong as these specialized bounds.
    This comparison however shows that our bound remains qualitatively similar up to polynomial factors.}
\end{remark}

\revision{
\begin{example}
    Consider the basis of tensorized Legendre polynomials $\vec{B}_{\vec{j}} = \bigotimes_{m=1}^M \vec{L}_{\vec{j}_m}$ and define the linear space $\mcal{V}_m$ in Theorem~\ref{thm:sparse_rip} as $\mcal{V}_m := \operatorname{span}\braces{\vec{B}_j : \omega_j^2 \le s}$.
    Then the bound in Theorem~\ref{thm:sparse_rip} depends on the parameter $s$ alone since the size of the hyperbolic cross
    \begin{equation*}
        \braces*{\vec{j}\in\mbb{N}^M : \omega_{\vec{j}}^2 \le s}
        \subseteq \braces*{\vec{j}\in\mbb{N}^M : \prod_{m=1}^M \pars{2\vec{j}_m + 1} \le s}
    \end{equation*}
    can be bounded by $m \lesssim s\log\pars{s}^{M-1}$ (cf.~\cite{rauhut2016weigtedl1}).
\end{example}
}

\revision{
According to Theorem~\ref{thm:pointwise_K}, the sampling density and weight function can be chosen optimally for a given model set.
For $\mcal{M}_{\omega,s}$ this is not straightforward because Lemma~\ref{thm:sparse_rip} bounds $K\pars{U\pars{\mcal{M}_{\omega,s}}} \le s$ independently of $w$ as long as $\omega_j \ge \norm{\vec{B}_j}_{w,\infty}$.
Note however that this bound is not unique since
\begin{equation}
    \mcal{M}_{\omega,s} = \mcal{M}_{c\omega, c^2s}
\end{equation}
for any $c>0$.
This means that $K\pars{U\pars{\mcal{M}_{\omega, s}}} \le c^2s$ for any $c$ that satisfies $c\omega_j \ge \norm{\vec{B}_j}_{w,\infty}$ and the smallest possible $c$ is given by $c_{\mathrm{min}} := \sup_{j\in\mbb{N}} \frac{\norm{\vec{B}_j}_{w,\infty}}{\omega_j}$.
An optimal weight function for the model class $\mcal{M}_{\omega,s}$ must thus minimize $c_{\mathrm{min}}$.
If we assume that $\omega \equiv \boldsymbol{1}$ then
\begin{equation}
    c_{\mathrm{min}}^2
    = \sup_{y\in Y} w\pars{y} \sup_{j\in\mbb{N}} \abs{\vec{B}_j}_y^2
    \le \sup_{y\in Y} w\pars{y} \hat{b}\pars{y}
    = \norm{w \hat{b}}_{L^\infty\pars{Y,\rho}} .
\end{equation}
From Theorem~\ref{thm:pointwise_K} we know that the minimum $\norm{w \hat{b}}_{L^\infty\pars{Y,\rho}} = \norm{\hat{b}}_{L^1\pars{Y,\rho}}$ is attained for the weight function $\tilde{w} = \norm{\hat{b}}_{L^1\pars{Y,\rho}}\hat{b}^{-1}$.
An upper bound for $\hat{b}$ and thus for $c_{\mathrm{min}}$ is computed in the subsequent Lemma~\ref{lem:weighted_sparsity_weight_function}.
The resulting sequence $\norm{\vec{B}_j}_{\tilde{w},\infty}$ is contrasted to the sequence $\norm{\vec{B}_j}_{w,\infty}$ for $j=1,\dots,100$ in Figure~\ref{fig:weight_sequences}.
We observe that the new weight function $\tilde{w}$ slightly increases $\norm{\vec{B}_1}_{\tilde{w},\infty}$ but considerably decreases $\norm{\vec{B}_j}_{\tilde{w},\infty}$ for all $j>1$.\footnote{This suggests a multi-level reconstruction where the $\vec{B}_1$-part is reconstructed with uniformly distributed samples and the remaining signal is reconstructed with respect to the newly computed sampling density and weight function. However, we have not tested this.}
A reconstruction using new weight function is shown in Figure~\ref{fig:reweighted_strongly_weighted_l1}.
Since the new constraint $\omega_j \ge \norm{\vec{B}_j}_{\tilde{w},\infty}$ is significantly weaker than the previous constraint $\omega_j \ge \norm{\vec{B}_j}_{\tilde{w},\infty}$ for $j>1$, one might ask what happens if the weight sequence $\omega$ is adapted as well.
Figure~\ref{fig:reweighted_weighted_l1} illustrates this for the smallest possible weight sequence $\omega_j := \norm{\vec{B}_j}_{\tilde{w},\infty}$.
Since this new weight sequence is almost constant (cf.\ Figure~\ref{fig:weight_sequences}) the resulting model class approximates the larger model class $\mcal{M}_{\boldsymbol{1},s}$.
This means that we can not expect the results in Figure~\ref{fig:reweighted_weighted_l1} to be better than those in Figure~\ref{fig:reweighted_strongly_weighted_l1}.
We observe however that they are indeed better than those in Figure~\ref{fig:unweighted_l1} where the model class $\mcal{M}_{\boldsymbol{1},s}$ was used.
}

\begin{lemma}\label{lem:weighted_sparsity_weight_function}
    Let $\mcal{V}_m$ be an $m$-dimensional subspace spanned by a subset of $\braces{\vec{B}_j}_{j\in\mbb{N}}$ and consider the model set $\mcal{M}_{\omega,s}\cap\mcal{V}_m$.
    Then
    \begin{equation}
        \hat{b}\pars{x} \le s \frac{\norm{\Omega^{-1}\vec{B}\pars{x}}_2^4}{\norm{\Omega^{-1}\vec{B}\pars{x}}_1^2}\quad\text{    with } \Omega := \operatorname{diag}\pars{\omega}.
    \end{equation}
\end{lemma}
\begin{proof}
    Observe that by the Cauchy-Schwarz inequality 
    \begin{equation}
        \norm{\Omega \vec{v}}_{1}
        = \sum_{j\in\operatorname{supp}\pars{\vec{v}}} \abs{\vec{v}_j} \omega_j
        \le \norm{\vec{v}}_{\omega, 0} \norm{\vec{v}}_2 .
    \end{equation}
    Defining the model set
    \begin{equation}
        \mcal{M} := \braces{v\in\mcal{V}_m : \norm{\Omega \vec{v}}_{1} \le \sqrt{s}\norm{\vec{v}}_2}
    \end{equation}
    we have the inclusion $\mcal{M}_{\omega,s}\cap\mcal{V}_m \subseteq \mcal{M}$.
    Since we know that $\hat{b}$ for $\mcal{M}_{\omega,s}\cap\mcal{V}_m$ is bounded by $\hat{b}$ for $\mcal{M}$, we derive an estimate for the larger set.
    
    Recall that
    \begin{equation}
        \hat{b}\pars{x} := \sup_{v\in\mcal{M}} \frac{\vec{v}^\intercal G\pars{x} \vec{v}}{\norm{\vec{v}}_{2}^2} \quad\text{with } G\pars{x} := \vec{B}\pars{x}\vec{B}\pars{x}^\intercal.
    \end{equation}
    Since $\norm{\vec{v}}_2^{-1} \le \sqrt{s}\norm{\Omega \vec{v}}_{1}^{-1}$ for all $v\in\mcal{M}$, we derive the bound
    \begin{equation}
        \hat{b}\pars{x} 
        \le s \sup_{v\in\mcal{M}} \frac{\vec{v}^\intercal G\pars{x} \vec{v}}{\norm{\Omega\vec{v}}_1^2}
        \le s \sup_{\substack{\vec{v}\in\mbb{R}^m \\ \vec{w} = \Omega \vec{v}}} \frac{\vec{w}^\intercal \Omega^{-1}G\pars{x}\Omega^{-1} \vec{w}}{\norm{\vec{w}}_{1}^2}
        = s \frac{\norm{\Omega^{-1}\vec{B}\pars{x}}_2^4}{\norm{\Omega^{-1}\vec{B}\pars{x}}_1^2} .
    \end{equation}
\end{proof}

\begin{figure}
    \centering
    \includegraphics[width=\textwidth]{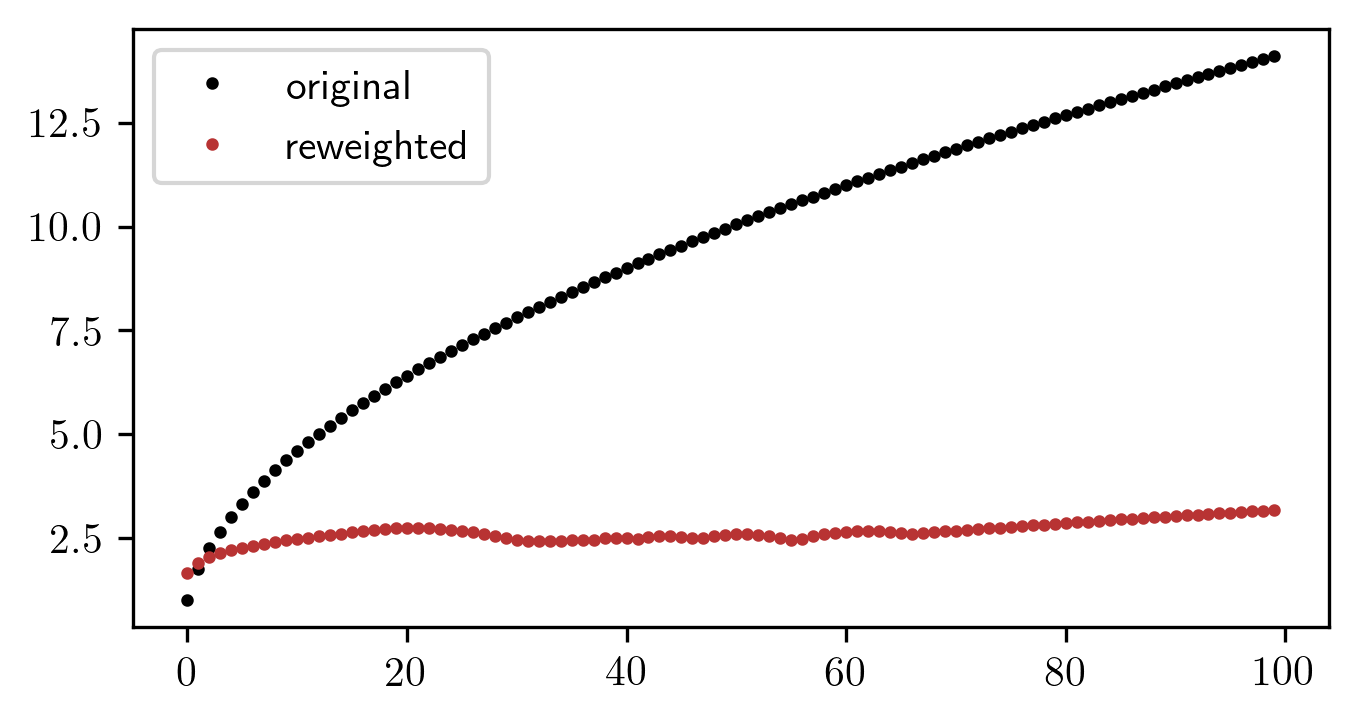}
    \caption{\revision{Let $\braces{\vec{B}_j}_{j\in\mbb{N}}$ be the Legendre polynomial basis on $\mcal{V} = L^2\pars{\bracs{0,1}, \tfrac{\dx{x}}{2}}$. The weight sequences for the original weight function are bounded by $\omega_{j} \ge \norm{\vec{B}_j}_{w, \infty}$ (black).
    The weight sequences for the adapted weight function $\tilde{w}$ (according to Theorem~\ref{thm:pointwise_K} and Lemma~\ref{lem:weighted_sparsity_weight_function}) are bounded by $\omega_{j} \ge \norm{\vec{B}_j}_{\tilde{w}, \infty}$ (red).}}
    \label{fig:weight_sequences}
\end{figure}

\begin{figure}
    \centering
    \begin{subfigure}[b]{0.3\textwidth}
        \centering
        \includegraphics[width=\textwidth]{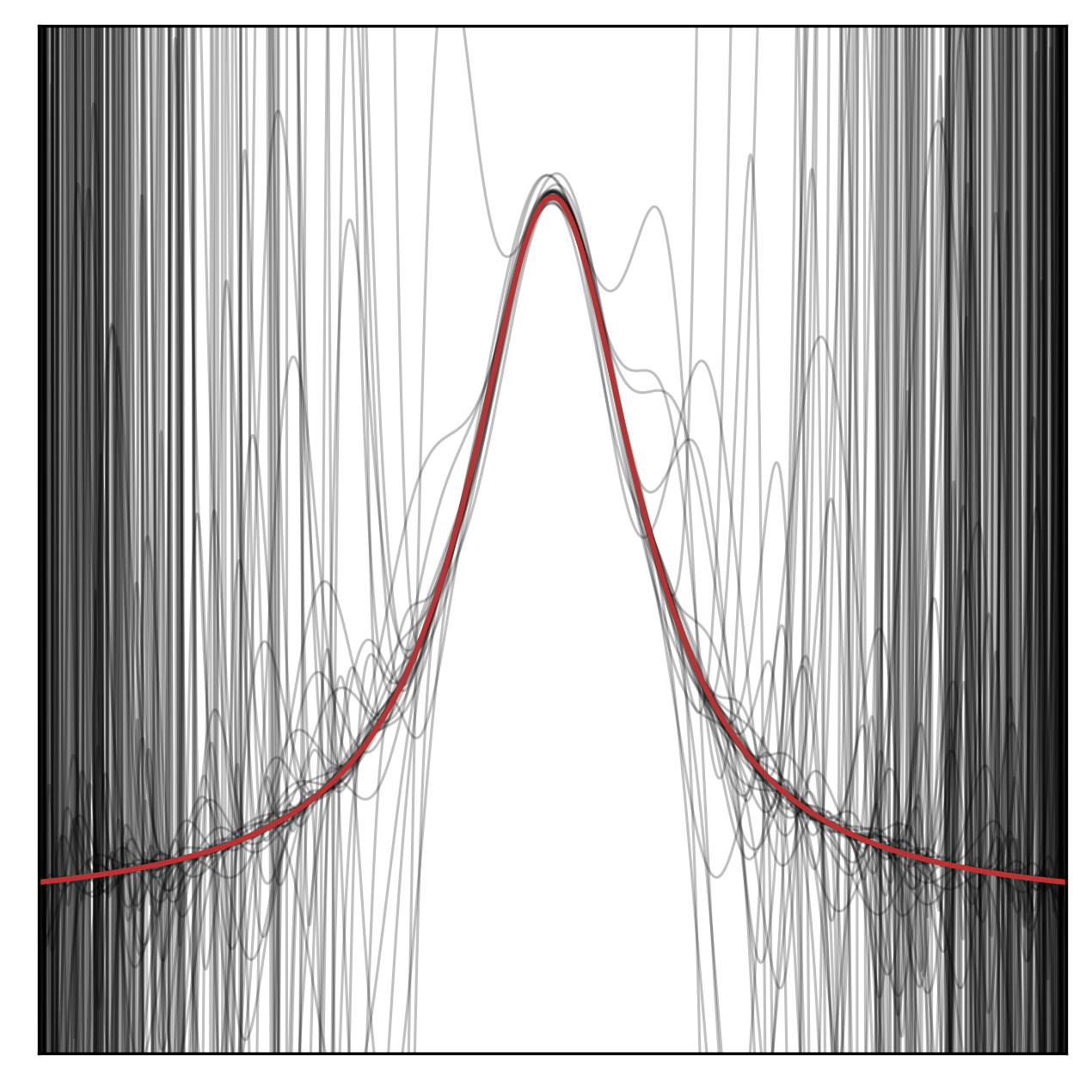}
        \caption{Exact inversion}
        \label{fig:exact_inversion}
    \end{subfigure}
    \hfill
    \begin{subfigure}[b]{0.3\textwidth}
        \centering
        \includegraphics[width=\textwidth]{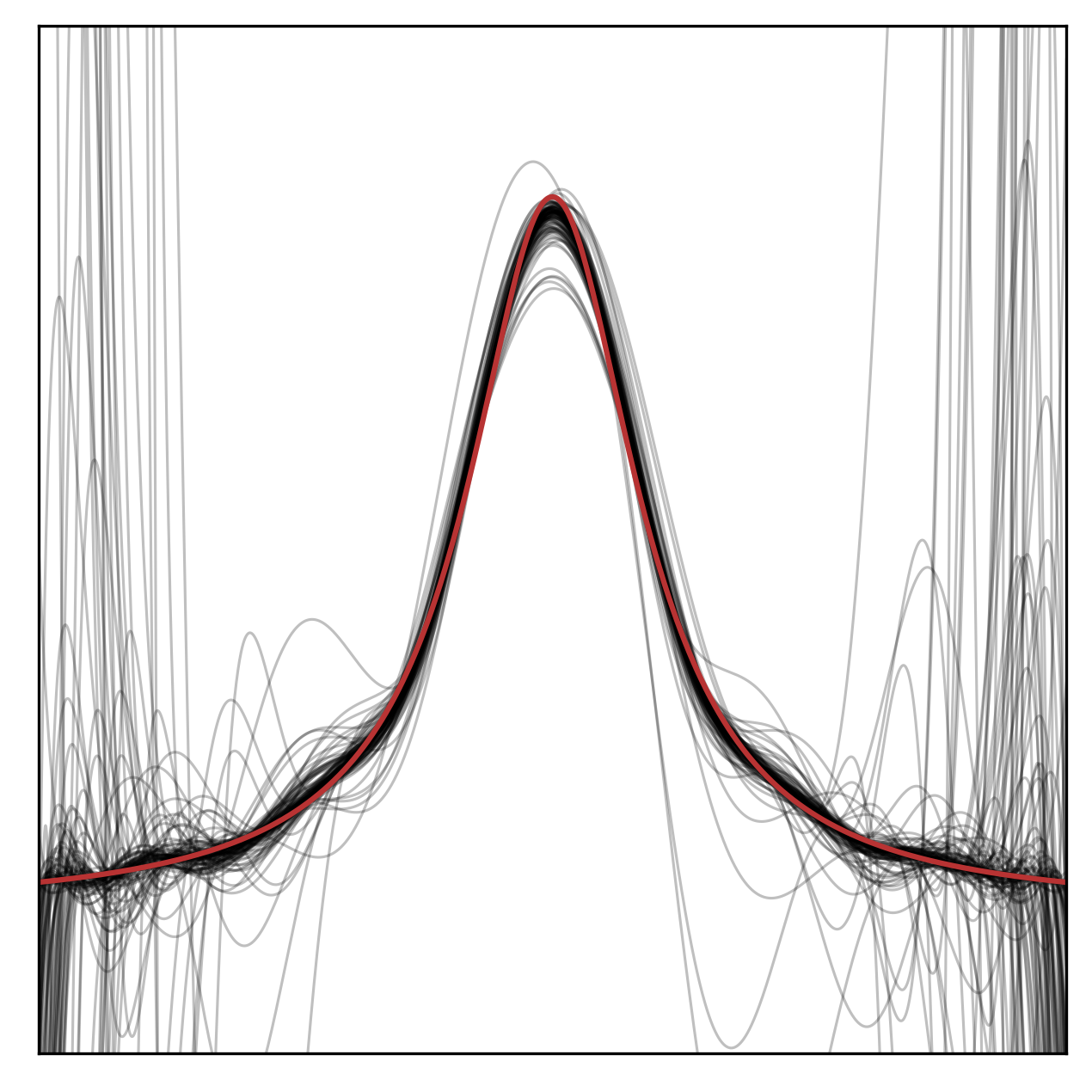}
        \caption{Least squares}
        \label{fig:least_squares}
    \end{subfigure}
    \hfill
    \begin{subfigure}[b]{0.3\textwidth}
        \centering
        \includegraphics[width=\textwidth]{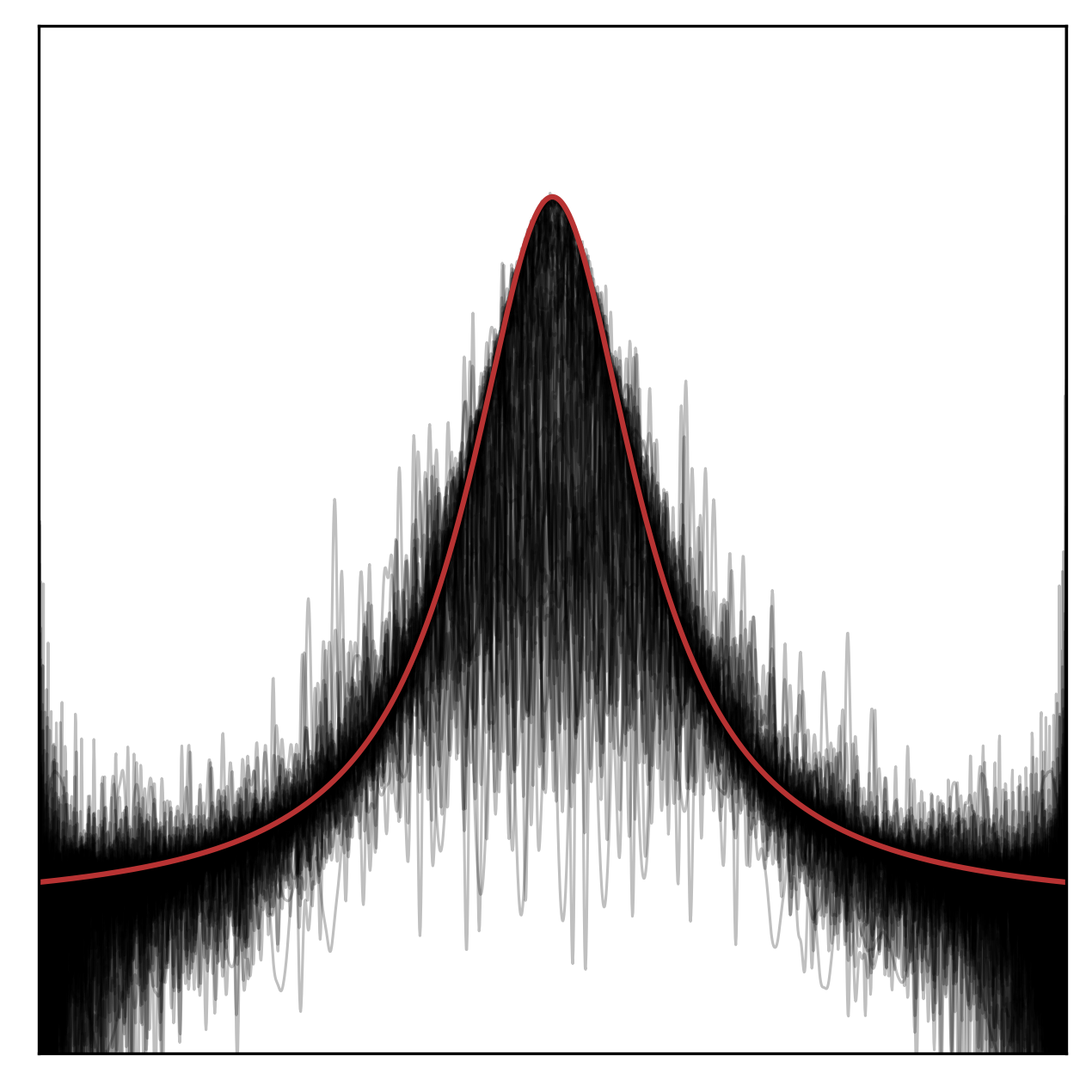}
        \caption{$\ell^1$, uniform $\omega$}
        \label{fig:unweighted_l1}
    \end{subfigure}
    \begin{subfigure}[b]{0.3\textwidth}
        \centering
        \includegraphics[width=\textwidth]{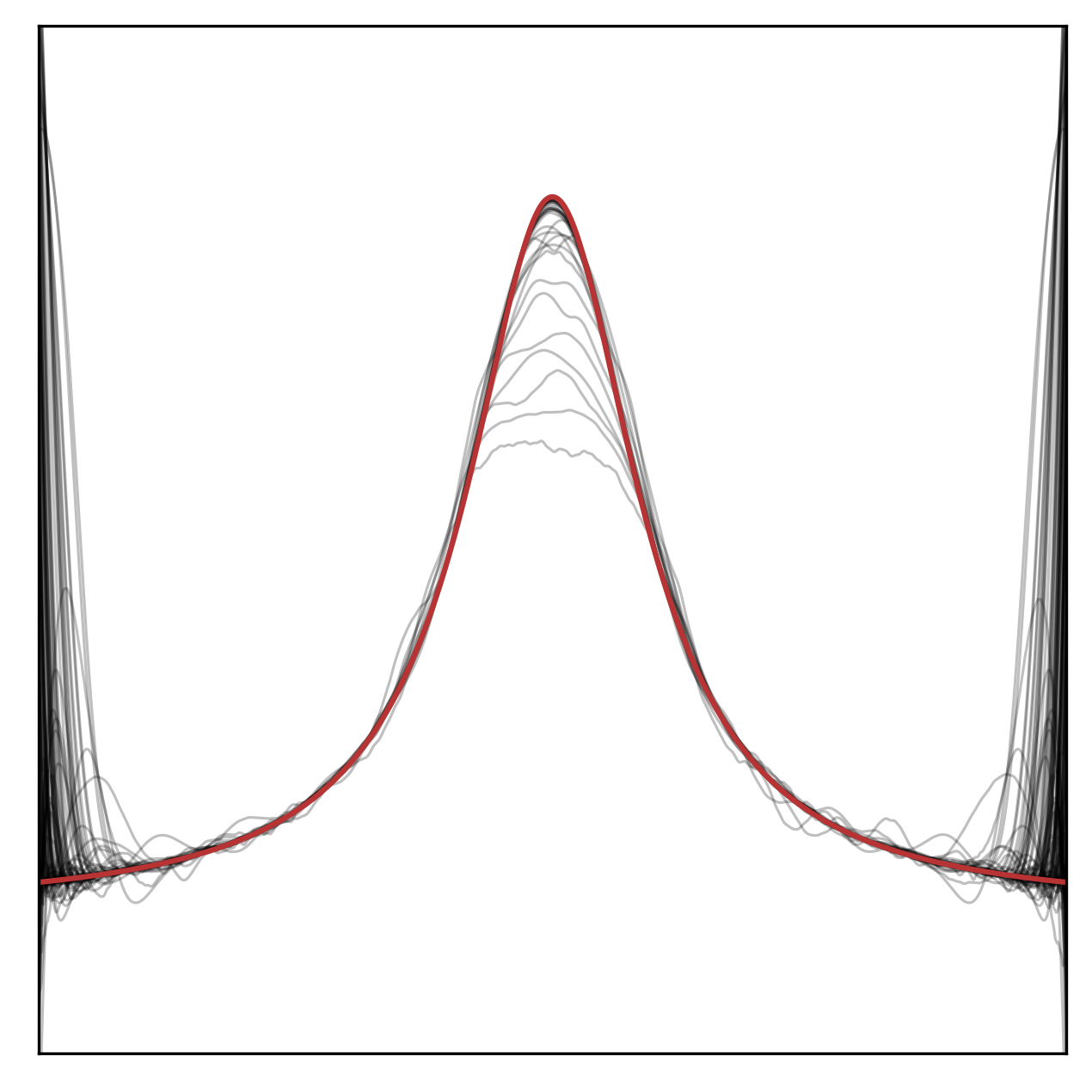}
        \caption{$\ell^1$, original $\omega$\\\phantom{Reweighted}}
        \label{fig:weighted_l1}
    \end{subfigure}
    \hfill
    \begin{subfigure}[b]{0.3\textwidth}
        \centering
        \includegraphics[width=\textwidth]{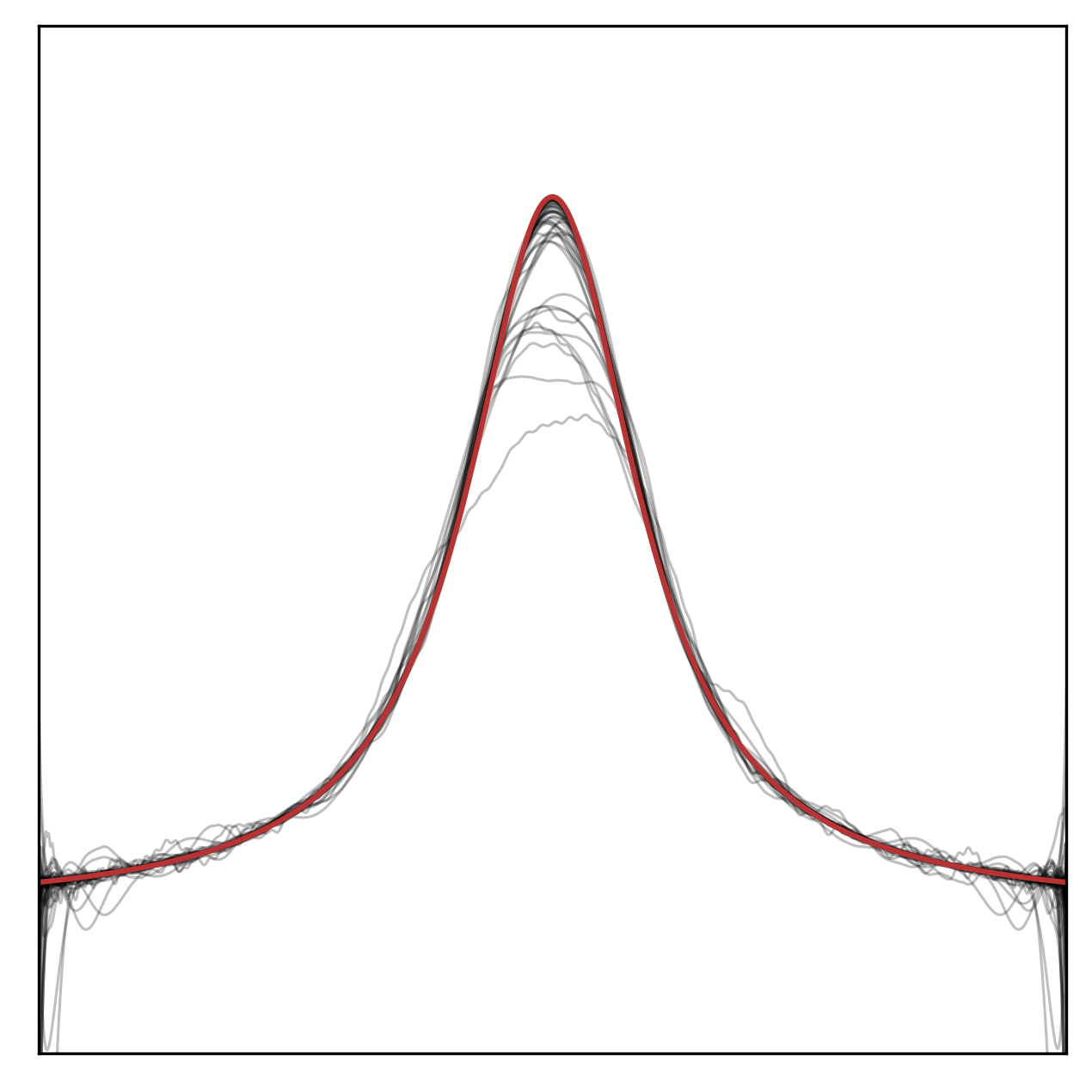}
        \caption{Reweighted $\ell^1$, original $\omega$}
        \label{fig:reweighted_strongly_weighted_l1}
    \end{subfigure}
    \hfill
    \begin{subfigure}[b]{0.3\textwidth}
        \centering
        \includegraphics[width=\textwidth]{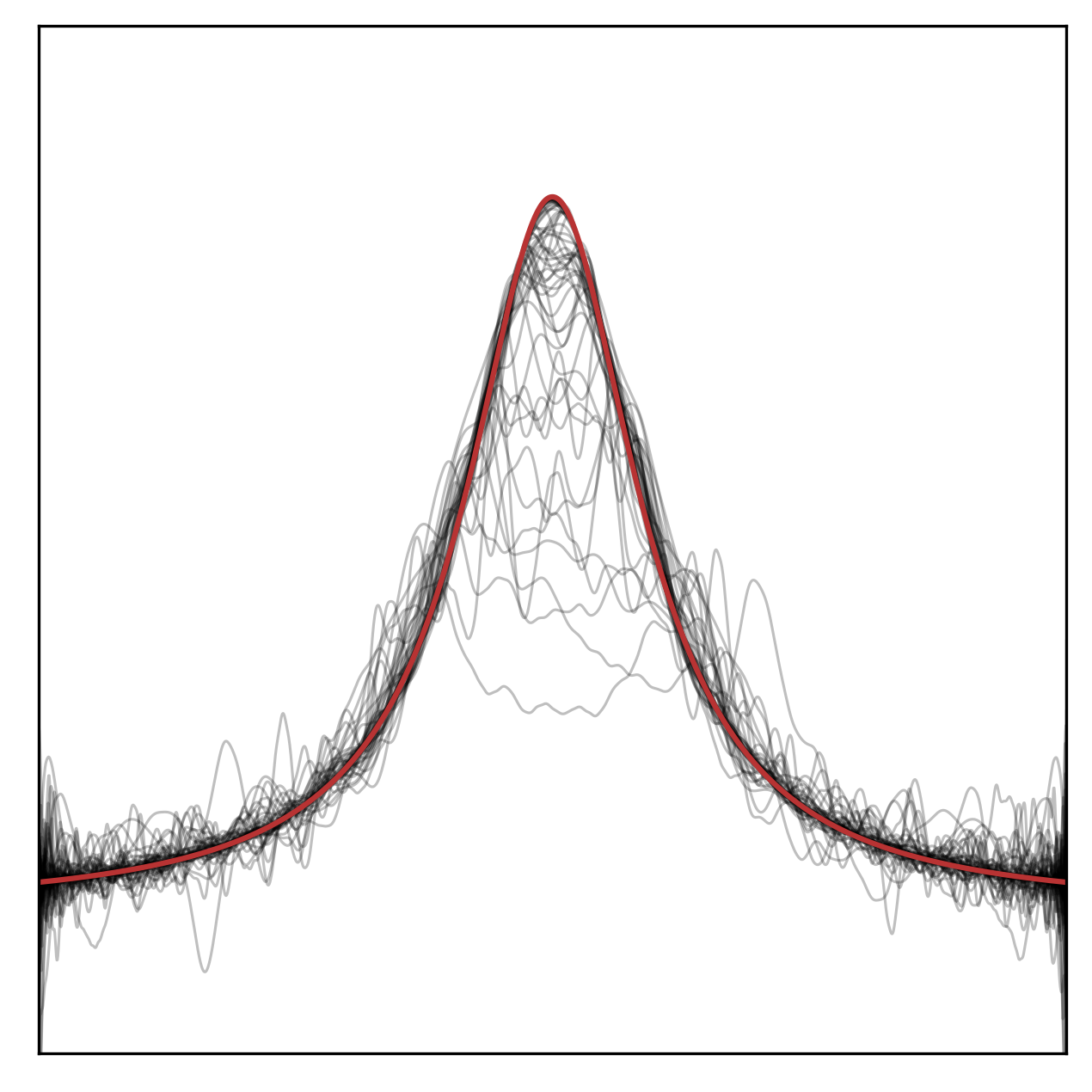}
        \caption{Reweighted $\ell^1$, minimal $\omega$}
        \label{fig:reweighted_weighted_l1}
    \end{subfigure}
    \captionsetup{singlelinecheck=off}
    \caption[.]{Overlaid interpolations of the function $f\pars{x} = \frac{1}{1+25x^2}$ (red) by Legendre polynomials \revision{of degree $99$} using various reconstruction methods.
    Different interpolations correspond to different random draws of $n=30$ sampling points.
    The subfigures~\ref{fig:exact_inversion} and~\ref{fig:least_squares} show \revision{standard} least squares \revision{approximations} of the first $30$ and $15$ basis functions respectively.
    \revision{The other figures employ the weighted $\ell^1$-minimization%
    \begin{displaymath}
        \hspace{-6em}\min_{\vec{v}\in\mbb{R}^{100}}\hspace{0.5em} \norm{\Omega \vec{v}}_1 \quad\text{s.t.}\quad \vec{v}^\intercal\vec{B}\pars{x_i} = y_i \quad \text{ for } 1\le i\le n,
    \end{displaymath}
    with $\Omega := \operatorname{diag}\pars{\omega}$.}
    \ref{fig:unweighted_l1} displays the results of unweighted $\ell^1$-minimization \revision{($\omega_j = 1$)} and \ref{fig:weighted_l1} displays the results of weighted $\ell^1$-minimization using the weight sequence $\omega_j = \norm{\vec{B}_j}_{L^\infty}$.
    In all aforementioned cases the sampling points are drawn according to the uniform measure on $\bracs{-1,1}$.
    The \revision{subplots~\ref{fig:reweighted_strongly_weighted_l1} and~\ref{fig:reweighted_weighted_l1}} use samples that are drawn according to the optimal sampling density as given in Lemma~\ref{lem:weighted_sparsity_weight_function}.
    \ref{fig:reweighted_strongly_weighted_l1} uses the original sequence $\omega_j = \norm{\vec{B}_j}_{L^\infty}$ while~\ref{fig:reweighted_weighted_l1} uses the minimal possible weight sequence $\omega_j = \norm{\vec{B}_j}_{\tilde{w},\infty} \lesssim \norm{\vec{B}_j}_{L^\infty}$.
    }
    \label{fig:weighted_l1_comparison}
\end{figure}

The theory presented in this subsection can be generalized easily to dictionary learning (cf.~\cite{Du2019dictionary_learning_chapter,jung2015dictionary_learning}).
This is stated without proof in the following theorem.
\begin{theorem}
    Assume that $\braces{\vec{B}_j}_{j\in\mbb{N}}$ is a Riesz sequence satisfying
    \begin{equation}
        c\norm{\vec{v}}_2^2 \le \norm*{\sum_{j\in\mbb{N}} \vec{v}_j\vec{B}_j}^2 \le C\norm{\vec{v}}_2^2
    \end{equation}
    \revision{and that $\omega$ is chosen such that $\omega_j \ge \norm{\vec{B}_j}$ for all $j$.
    Redefine}
    \begin{equation}
        \mcal{M}_{\omega,s} := \braces{v\in\mcal{V} : \exists\vec{v} \ \text{ s.t. }\ v = \sum_{j=1}^\infty \vec{v}_j\vec{B}_j \wedge \norm{\vec{v}}_{\omega,0} \le s}
    \end{equation}
    \revision{and let} $\mcal{V}_m\subset\mathcal V$ be an $m$-dimensional subspace spanned by a subset of $\braces{\vec{B}_j}_{j\in\mbb{N}}$.
    Then it holds that
    \begin{itemize}
        \item $\mcal{M}_{\omega, s} \subseteq \mcal{M}_{\omega, t}$ for $s\le t$,
        \item $\mcal{M}_{\omega,s} = -\mcal{M}_{\omega,s}$,
        \item $\mcal{M}_{\omega,s}+\mcal{M}_{\omega,t} \subseteq \mcal{M}_{\omega,s+t}$,
        \item $\norm{v}_{w,\infty} \le \frac{\sqrt{s}}{c}\norm{v}$ for all $v\in\mcal{M}_{\omega, s}$ and
        \item $\nu_{\norm{\bullet}_{w,\infty}}\pars{U\pars{\mcal{M}_{\omega,s}\cap \mcal{V}_m}, r} \revision{\le} \pars*{\tfrac{\revision{2k}Cm}{r\sqrt{s}}}^{s}$ \revision{for some $k>0$}.
    \end{itemize}
\end{theorem}

\subsection{Tensors of rank \texorpdfstring{$r$}{r}} \label{sec:cp-tensors}

\revision{We now consider two different problems related to sets of low-rank tensors.
Both cases can be expressed with $Y = \pars{\mbb{R}^{m}}^{\otimes M}$ and $\mcal{V} = \pars{Y, \norm{\bullet}}$ with $\abs{v}_y := \abs{\pars{v, y}_{\mathrm{Fro}}}$.
The only difference is the distribution $\rho$ from which the samples are drawn.
{
\renewcommand{\labelenumi}{(\arabic{enumi})}
\begin{enumerate}
    \item \textit{Recovery from Gausian samples.}
    In this problem $\rho = \mcal{N}\pars{0_Y, \operatorname{Id}_Y}$ is a Gaussian distribution on the entire space $Y$.
    Although this problem is rather artificial it was one of the first where rigorous bounds were developed in~\cite{RAUHUT2017220}.
    \item \textit{Recovery from rank-$1$ samples and completion.}
    For this problem let $\braces{\rho_k}_{k\in\bracs{M}}$ be distributions on $\mbb{R}^m$ and consider $\rho = \rho_1\otimes\cdots\otimes\rho_M$.
    This problem occurs for example whenever one tries to approximate a low-rank function of $M$ variables using a tensor product basis.
    A special case of this setting is the problem of tensor completion where a tensor has to be  recovered from a few entries.
    In this problem all distributions $\rho_k$ have to be discrete measures on the standard basis vectors.
\end{enumerate}
}
In both problems the task is to find a best approximation in a subset $\mcal{T}_r\subseteq\mcal{V}$ of bounded rank $r$.
For tensors however there exist many different concepts of rank for which we refer to~\cite{bachmayr-schneider,hackbusch2012tensorBuch,grasedyck2011introduction,hitchcock1927cp_format} and the works cited below.
}

\revision{
\subsubsection*{Recovery from Gaussian samples}

In this section we consider a subset $\mcal{T}_r\subseteq\mcal{V}$ of tensors of bounded (Hierarchical Tucker) HT-rank $r$.
For $w\equiv 1$ the following bound for the sample complexity subject to $\delta$ is given in~\cite[Theorem 2]{RAUHUT2017220},
\begin{equation}
    n \gtrsim \max\braces{\pars{\pars{M-1}r^3+Mmr}\ln\pars{Mr}, \ln\pars{p^{-1}}}\delta^{-2} .
\end{equation}

To obtain a sample bound from our theory, we would have to bound the variation constant, which however is infinity,
\begin{equation}
    K\pars{U\pars{\mcal{T}_r}}
    = \sup_{\substack{v\in\mcal{T}_r\\\norm{v}=1}} \esssup_{y\in Y} \abs{\pars{v, y}_{\mathrm{Fro}}}^2
    = \infty .
\end{equation}
This shows that a direct application of the presented formalism to this problem cannot provide a finite sample complexity.

\begin{remark}
    As above, this exposes the lacking sharpness of the results used in the proof of Theorem~\ref{thm:rip}.
    With more refined concentration inequalities as in~\cite{dirksen2015generic_chaining}, a different definition of the variation constant would emerge (replacing $\norm{\bullet}_{w,\infty}$ by a sub-Gaussian norm), which would be finite for this problem.
\end{remark}
}

\revision{
The present theory can deal with this problem in two different ways.
The first option is to choose the weight function $w\pars{y} = m^M\norm{y}_{\mathrm{Fro}}^{-2}$, which yields the variation constant
\begin{equation}
    K\pars{U\pars{\mcal{T}_r}}
    = m^M \sup_{\substack{v\in\mcal{T}_r\\\norm{v}=1}} \esssup_{\substack{y\in Y\\\norm{y}=1}} \abs{\pars{v, y}_{\mathrm{Fro}}}^2
    = m^M,
\end{equation}
where the final equality holds since $\norm{\bullet} = \norm{\bullet}_{\mathrm{Fro}}$.
The second option is to normalize the samples and thereby replace the Gaussian distribution by a uniform distribution on the unit sphere.
In this case we obtain the new identity $\norm{\bullet} = m^{-M/2}\norm{\bullet}_{\mathrm{Fro}}$ and the corresponding variation constant
\begin{equation}
    K\pars{U\pars{\mcal{T}_r}}
    = m^M \sup_{\substack{v\in\mcal{T}_r\\\norm{v}=1}} \esssup_{\substack{y\in Y\\\norm{y}=1}} \abs{\pars{v, y}_{\mathrm{Fro}}}^2
    = m^M .
\end{equation}
In both cases $K\pars{U\pars{\mcal{T}_r}} = K\pars{U\pars{\mcal{V}}}$.

Let $k = \sqrt{K\pars{U\pars{\mcal{T}_{2r}}}}$.
By using the bound $\norm{\bullet}_{w,\infty} \le k\norm{\bullet}$ on $\mcal{T}_{2r}$ we can utilize the bound for the covering number for tensors of HT-rank $r$ that is provided in~\cite{RAUHUT2017220}.
This leads to the estimate
\begin{equation}
    \nu_{\norm{\bullet}_{w,\infty}}\pars{U\pars{\mcal{T}_r}, \varepsilon}
    \le \nu_{\norm{\bullet}}\pars*{U\pars{\mcal{T}_r}, \frac{\varepsilon}{k}}
    \le \pars*{\frac{\varepsilon}{3\pars{2M-1}\sqrt{r}k}}^{-\pars{Mr^3 + Mmr}} .
\end{equation}
A subsequent application of Corollary~\ref{cor:sample_size_rip} yields
\begin{equation}
    n \le 2 \pars*{\pars{Mr^3 + Mmr}\ln\pars{3\pars{2M-1}\sqrt{r}k\delta^{-1}} - \ln\pars*{\frac{p}{2}}}\pars*{\frac{k^2}{\delta}}^2 .
\end{equation}
For $k=1$ this has the same asymptotic complexity as the bound in~\cite{RAUHUT2017220}.
We conjecture that the transition $k=m^{M/2} \leadsto k=1$ can be achieved by using a generic chaining argument (cf.~\cite{dirksen2015generic_chaining}) rather than a simple Hoeffding bound in the proof of Theorem~\ref{thm:rip}.
}

\revision{
\subsubsection*{Recovery from rank-$1$ samples and completion}

In this section we consider subsets $\mcal{T}_r\subseteq\mcal{V}$ of generic rank-$r$ tensors but assume that the rank concept satisfies $\mcal{T}_1\subseteq\mcal{T}_r$.
This is the case for all tree-shaped tensor formats including the Tucker format, the tensor train (TT) format and general hierarchical tensor formats (HT) as well as the canonical polyadic decomposition (CP).
For the sake of completeness we define
\begin{equation}
    \mcal{T}_1 := \braces{v\in\mcal{V} : \vec{v} = \vec{v}_1\otimes\cdots\otimes\vec{v}_M \text{ with } \vec{v}_1,\ldots,\vec{v}_M\in\mbb{R}^m} .
\end{equation}
The variation constant for the set $\mcal{T}_r$ is computed in the next theorem.

\begin{theorem} \label{thm:K_tensor}
    Assume that $y$ has rank $1$ almost surely.
    Then $K\pars{U\pars{\mcal{T}_r}} = K\pars{U\pars{\mcal{V}}}$.
\end{theorem}
\begin{proof}
    Observe that
    \begin{align}
        \hat{b}_{\mcal{V}}\pars{y}
        &:= \sup_{\substack{v\in\mcal{V}\\\norm{v}=1}} \abs{\pars{v,y}_{\mathrm{Fro}}}^2
        = \abs*{\pars*{\tfrac{y}{\norm{y}},y}_{\mathrm{Fro}}}^2
        = \frac{\norm{y}_{\mathrm{Fro}}^4}{\norm{y}^2}. \\
    \intertext{Since $y$ has rank $1$,}
        \hat{b}_{\mcal{T}_1}\pars{y}
        &:= \sup_{\substack{v\in\mcal{T}_1\\\norm{v}=1}} \abs{\pars{v,y}_{\mathrm{Fro}}}^2
        = \abs*{\pars*{\tfrac{y}{\norm{y}},y}_{\mathrm{Fro}}}^2
        = \frac{\norm{y}_{\mathrm{Fro}}^4}{\norm{y}^2}.
    \end{align}
    We deduce that $K\pars{U\pars{\mcal{T}_1}} = K\pars{U\pars{\mcal{V}}}$ by Theorem~\ref{thm:pointwise_K}.
    This proves the assertion since $\mcal{T}_1 \subseteq \mcal{T}_r \subseteq \mcal{V}_m^{\otimes M}$ implies $K\pars{U\pars{\mcal{T}_1}} \le K\pars{U\pars{\mcal{T}_r}} \le K\pars{U\pars{\mcal{V}}}$.
\end{proof}
The theorem states that tensor formats do \textbf{not} exhibit a smaller variation constant than the linear space they are embedded in.
This result is surprising at first because tensor formats have a significantly smaller covering number than the full tensor space, cf.~\cite{RAUHUT2017220}.
However, this is already indicated by the classical analysis of matrix completion from which it is known that the notion of incoherence is required in addition to a low-rank property.
}

\revision{
Despite this unfavourable result, it is noteworthy that the present theory can be used in this setting.
The bound $\norm{\bullet}_{w,\infty} \le \sqrt{K\pars{U\pars{\mcal{T}_{2r}}}}\norm{\bullet}$ and the isometry $\norm{\bullet} = \norm{\bullet}_{\mathrm{Fro}}$ imply
\begin{equation}
    \nu_{\norm{\bullet}_{w,\infty}}\pars{U\pars{\mcal{T}_r}, \varepsilon}
    \le \nu_{\norm{\bullet}_{\mathrm{Fro}}}\pars*{U\pars{\mcal{T}_r}, \frac{\varepsilon}{\sqrt{K\pars{U\pars{\mcal{T}_{2r}}}}}} .
\end{equation}
Assuming the weight function $w$ is chosen optimally, we know from Theorem~\ref{thm:K_tensor} and Section~\ref{sec:linear} that $K\pars{U\pars{\mcal{T}_{2r}}} = m^M$.
We can now apply the bound for the covering number of tensors of HT-rank $r$ from~\cite{RAUHUT2017220}.
The resulting estimate reads
\begin{equation}
    \nu_{\norm{\bullet}_{w,\infty}}\pars{U\pars{\mcal{T}_r}, \varepsilon}
    \le \pars*{\frac{\varepsilon}{3\pars{2M-1}\sqrt{rm^M}}}^{-\pars{Mr^3 + Mmr}} .
\end{equation}
A final application of Corollary~\ref{cor:sample_size_rip} yields
\begin{equation}
    n \le 2 \pars*{\pars{Mr^3 + Mmr}\ln\pars{3\pars{2M-1}\sqrt{r m^M}\delta^{-1}} - \ln\pars*{\frac{p}{2}}}\pars*{\frac{m^M}{\delta}}^2 .
\end{equation}
To the knowledge of the authors this is the first estimate of the number of samples that are necessary to satisfy \rip{\mcal{T}_r}{\delta} in this setting. %
Note that this is a worst-case estimate and that significantly less samples are needed in practice (cf.~\cite{ESTW19}).
}

\revision{
In the following examples we discuss the application to two common classes of problems.
\begin{example} \label{ex:rk1-measurements}
    In this example we consider the problem of recovering the low-rank coefficient tensor of a function from samples.
    Let $\pi_m$ be a probability measure on $Z_m$ and $\mcal{W}_m \subseteq L^2\pars{Z_m,\pi_m}$ be spanned by the $d_m$ \emph{orthonormal} basis functions $\braces{\vec{B}_{m,j}}_{j\in\bracs{d_m}}$.
    Now define the product space $\mcal{W} := \mcal{W}_m^{\otimes M} \subseteq L^2\pars{Z, \pi}$ with $Z := Z_m^{M}$ and $\pi := \pi_m^{\otimes M}$ and endow it with the seminorm $\absDagger{w}_z := \abs{w\pars{z}}$.
    This is the space in which the sought functions will live and it shall be approximated in the norm $\normDagger{\bullet} := \norm{\bullet}_{L^2\pars{Z,\pi}}$.
    As a model class consider the set $\mcal{T}_r^{\mcal{W}}\subseteq\mcal{W}$ of functions with a coefficient tensor of rank $r$ with respect to the tensor product basis $\vec{B}_{\vec{j}}\pars{z} := \prod_{k=1}^M\vec{B}_{m,\vec{j}_k}\pars{z_k}$ and denote this set of coefficient tensors by $\mcal{T}_r^{\mcal{V}} := \mcal{T}_r$.
    For the sake of simplicity, assume that the weight function $w\equiv 1$ is constant.
    
    To compute the variation constant of this model class, recall the definition of $\mcal{V} = \pars{Y, \norm{\bullet}}$, $Y = \pars{\mbb{R}^m}^{\otimes M}$ and $\abs{v}_y = \abs{\pars{v,y}_{\mathrm{Fro}}}$ from above.
    Note that each function $w\in\mcal{W}$ corresponds uniquely to a coefficient tensor $\vec{w}\in\mcal{V}$ and that the mapping $\vec{B} : Z \to Y$ given by $\pars{\vec{B}\pars{z}}_{\vec{j}} := \vec{B}_{\vec{j}}\pars{z}$ induces an isometry of seminorms
    \begin{equation}
        \absDagger{w}_{z} = \abs{w\pars{z}} = \abs{\pars{\vec{w},\vec{B}\pars{z}}_{\mathrm{Fro}}} = \abs{\vec{w}}_{\vec{B}\pars{z}} .
    \end{equation}
    This means that if we choose $\rho$ as the pushforward measure $\rho := B_*\pi$ the isometry of seminorms induces the isometry of the two norms
    \begin{align}
        &\norm{\vec{w}}
        = \pars*{\int_{Y} \abs{\vec{w}}_y^2 \dmx{\rho}{y}}^{1/2}
        = \pars*{\int_{Z} \absDagger{w}_z^2 \dmx{\pi}{z}}^{1/2}
        = \normDagger{w}
    \intertext{and}
        &\norm{\vec{w}}_{1, \infty}
        = \esssup_{y\in Y} \abs{\vec{w}}_y
        = \esssup_{z\in Z} \absDagger{w}_z
        =: \normDagger{w}_{1,\infty} .
    \end{align}
    Together with Theorem~\ref{thm:K_tensor} and Theorem~\ref{thm:pointwise_K} it follows that
    \begin{equation*}
        K\pars{U\pars{\mcal{T}^{\mcal{W}}_r}} 
        = K\pars{U\pars{\mcal{T}^{\mcal{V}}_r}}
        = K\pars{U\pars{\mcal{V}}} 
        \ge m^M .
    \end{equation*}
    This shows that the variation constant for this model class grows exponentially with $M$.
\end{example}
}

\revision{
\begin{example}
    The problem of tensor completion can be considered as a special case of Example~\ref{ex:rk1-measurements}.
    In this setting $Z = \bracs{m}^M$ is the set of all multi-indices, $\pi=\mcal{U}\pars{Z}$ is a uniform distribution on $Z$ and $\mcal{W} = \pars{\pars{\mbb{R}^m}^{\otimes M}, \norm{\bullet}_{\mathrm{Fro}}}$ is endowed with the semi-norm $\absDagger{w}_z := m^M\abs{w_z}$.
    Since this is a special case of Example~\ref{ex:rk1-measurements} the model class of rank-$r$ tensors $\mcal{T}_r$ exhibits the same bound $K\pars{U\pars{\mcal{T}_r}} = K\pars{U\pars{\mcal{W}}} \ge m^M$.
\end{example}
}

\revision{
These two examples show that $K\pars{U\pars{\mcal{T}_r}} \ge m^M$ in important applications.
To reduce the variation constant in these cases we can only intersect $\mcal{T}_r$ with another model class $\mcal{M}$ with low variation constant.
The intersection then inherits the low covering number of $\mcal{T}_r$ and the low variation constant of $\mcal{M}$.
}

\revision{
\section{Dependence on the seminorm}
\label{sec:Hk_seminorm}

Since the definition of the $\norm{\bullet}$-norm is very general, our theory is not limited to the $L^2$-norm but extends to Sobolev or energy norms. 
It is therefore natural to ask how the choice of the semi-norm $\abs{\bullet}_y$ influences the variation constant.
In this section we investigate this influence using Sobolev norms as an example.

We will need the following generalization of reproducing kernel Hilbert spaces (GRKHS) as a tool for the analysis.
\begin{definition}[Generalized Reproducing Kernel Hilbert Space]
    Let $\mcal{H}\subseteq\mcal{V}$ and $\braces{L_y}_{y\in Y} \subseteq \mcal{L}\pars{\mcal{H}, \mbb{R}^\ell}$ be a family of bounded $y$-dependent linear operators.
    Then the pair $\pars{\mcal{H}, \braces{L_y}_{y\in Y}}$ generalizes the concept of reproducing kernel Hilbert spaces.
\end{definition}
If $\pars{\mcal{H}, \braces{L_y}_{y\in Y}}$ forms a GRKHS and $\abs{v}_y := \norm{L_yv}_2$ then
\begin{equation}
    \norm{v}_{w,\infty} \le \varkappa \norm{v}_{\mcal{H}}
    \qquad\text{and}\qquad
    K\pars{U\pars{A}} \le \varkappa^2 \lambda^2,
\end{equation}
for $v\in A\subseteq\mcal{H}$ with $\varkappa := \sup_{y\in Y} \sqrt{w\pars{y}} \norm{L_y}_{\mcal{L}\pars{\mcal{H},\mbb{R}^l}}$ and 
$\lambda := \sup_{v\in A\setminus\braces{0}} \frac{\norm{v}_{\mcal{H}}}{\norm{v}}$.
This allows to efficiently compute an upper bound for $K\pars{U\pars{A}}$ even if the dimension of $Y$ is large.
\begin{remark}
    In this setting the application of Theorem~\ref{thm:error_bound} leads to
    \begin{equation}
        \norm{u - \epprum}
        \lesssim \norm{u - \apprum}_{w,\infty}
        \le \varkappa \norm{u - \apprum}_{\mcal{H}}
    \end{equation}
    whenever \rip{\braces{\apprum} - \mcal{M}}{\delta} holds.
\end{remark}
In the following we consider a linear model space $\mcal{M}\subseteq\mcal{H} := H^{M}\pars{Y}$ with a Lipschitz domain $Y\subseteq \mbb{R}^d$.
For each $m \le M - \tfrac{d}{2}$ we consider $\mcal{V} := H^m\pars{Y}$ with $\abs{v}_y := \norm{L_y v}_2$ and $L^m_y \in \mcal{L}\pars{\mcal{H}, \mbb{R}^\ell}$.
This means that we are searching the best approximation in the model space $\mcal{M}$ with respect to the $H^m$-norm.
To investigate the influence of $m$ on the sample complexity, the upper bound $\varkappa_m^2\lambda_m^2$ for $K\pars{U\pars{\mcal{M}}}$ depending on $m$ has to be computed.

It is proved in Appendix~\ref{app:sec:Hk_seminorm} that for $w\equiv 1$
\begin{equation}
    \varkappa_m := \pars{2\sqrt{\pi}}^{-d} \frac{\Gamma\pars{M+1} \Gamma\pars{M-m-\frac{d}{2}}}{\Gamma\pars{M-m}} .
\end{equation}
Since $\varkappa_m$ increases but $\lambda_m$ decreases with $m$, both effects should be equilibrated by a proper choice of $m$.
This is illustrated for two different model spaces $\mcal{M}$ in Figure~\ref{fig:rkhs_theoretical}.
The small effect of $\varkappa_m$ is due to the dimension $d=1$ for which we can bound
\begin{equation}
    \frac{\pars{M+1}!}{2\sqrt{\pi}} \pars{M-m}^{-1/2}
    < \varkappa_m
    < \frac{\pars{M+1}!}{2\sqrt{\pi}} \pars{M-m-1}^{-1/2}
\end{equation}
by Gautschi's inequality~\cite[Eq.~5.6.4]{DLMF}.

We conclude that for linear model spaces an approximation with respect to the $H^m$-norm for larger $m$ requires less samples than an approximation with respect to the $L^2$-norm.
For $m=1$ this hypothesis is confirmed numerically in Figure~\ref{fig:H1_vs_L2_inf}.
For an application in the setting of weighted sparsity we refer to the recent work~\cite{adcock2019H1_weighted_sparsity}.
Note that this does not have to be the case in general.
If the model class contains only piecewise constant functions then information about the gradients is irrelevant.
Such phenomena may also arise due to intricate properties of the model class and may only be observable by looking at the variation constant.

\begin{figure}
    \centering
    \includegraphics[width=\textwidth]{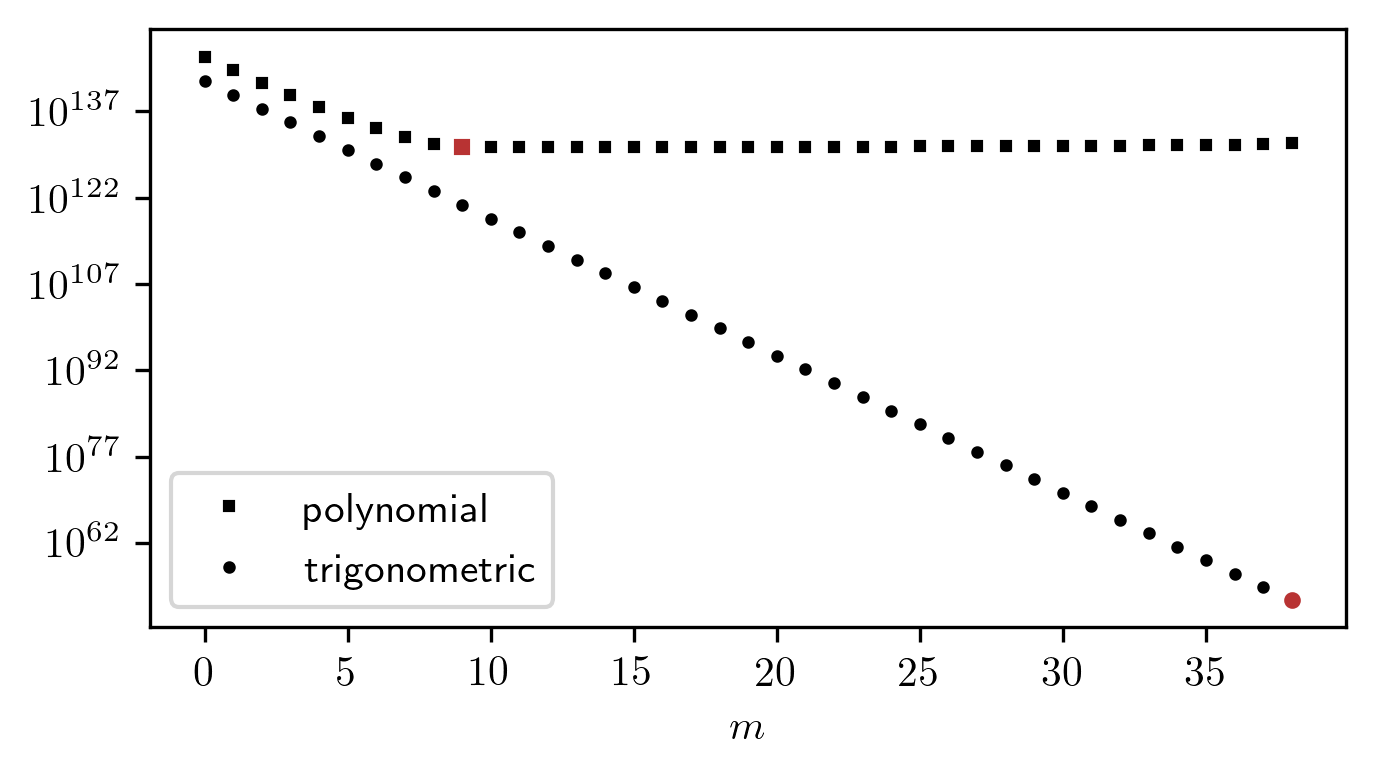}
    \caption{The upper bound $\varkappa_m^2\lambda_m^2$ for the variation constant \revision{for two different model spaces $\mcal{M}$ with $Y=[-1,1]$ and $M=40$.
    The squares and dots represent the bound when $A$ is the span of the first $10$ polynomials and trigonometric polynomials, respectively.}
    The optimal $m$ is marked fat and in red.}
    \label{fig:rkhs_theoretical}
\end{figure}
\begin{figure}
    \centering
    \begin{subfigure}[b]{0.475\textwidth}
        \centering
        \includegraphics[width=\textwidth]{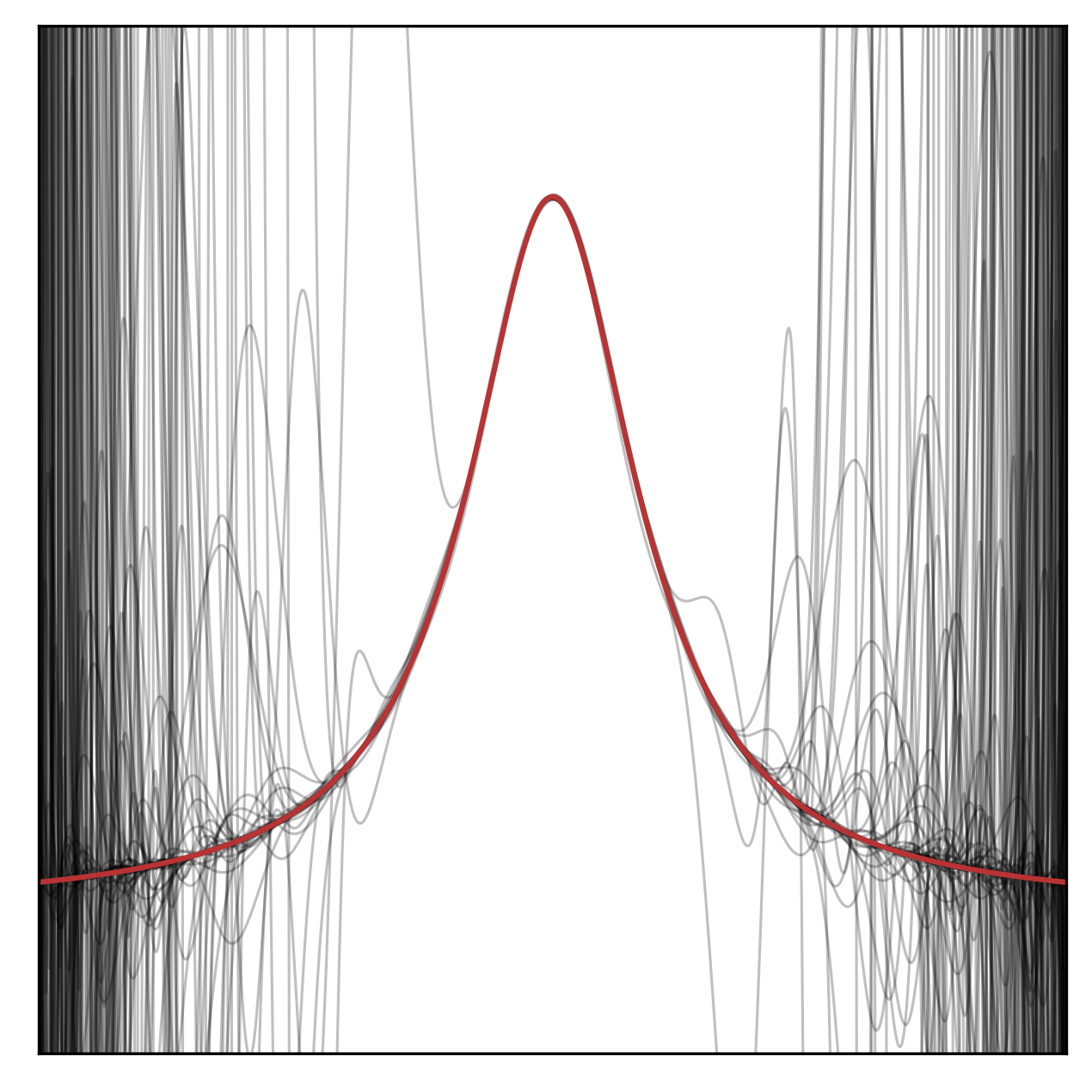}
        \caption{Least squares, $L^2$-norm}
        \label{fig:least_squares_L2}
    \end{subfigure}
    \hfill
    \begin{subfigure}[b]{0.475\textwidth}
        \centering
        \includegraphics[width=\textwidth]{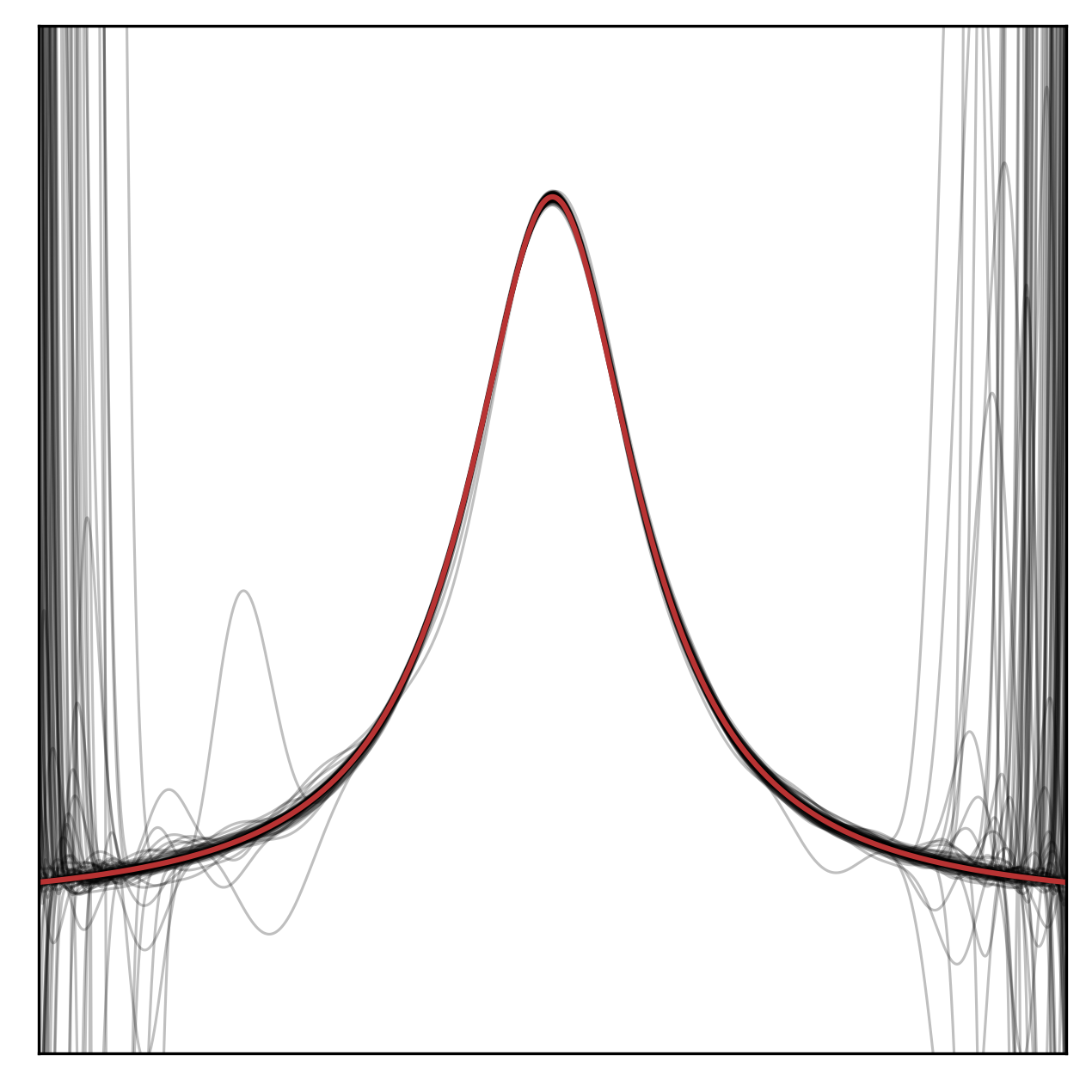}
        \caption{Least squares, $H^1$-norm}
        \label{fig:least_squares_H1}
    \end{subfigure}
    \caption{
    Overlaid least squares approximations of the function $f\pars{x} = \frac{1}{1+25x^2}$ (red) by Legendre polynomials of degree $29$.
    Different approximations correspond to different random draws of $n=40$ sampling points from the uniform measure on $\bracs{-1,1}$.
    }
    \label{fig:H1_vs_L2_inf}
\end{figure}

Also note that the minimization with respect to the $H^m$-norm does not necessarily require more computational effort than the minimization with respect to the $L^2$-norm.
The values of both seminorms can be computed with a single evaluation of the Fourier transform $\hat{u} = \mcal{F} u$ of $u$.
A particularly important application of this setting is Magnetic Resonance Imaging (MRI).
Recalling Remark~\ref{rmk:residuals}, we describe this application in the following example.
\begin{example}[MRI] \label{ex:MRI}
    In Magnetic Resonance Imaging an image $u$ is sampled via evaluations of its Fourier transform $\hat{u} = \mcal{F}u$.
    This means that the samples $\braces{\hat{u}_i}_{i\in\bracs{n}}$ satisfy $\hat{u}_i = \hat{u}\pars{\omega_i}$ for samples of the angular frequency $\omega_i$.
    The precise distribution of the samples $\omega_i$ is given by the problem and is not of particular interest in this example.
    Since $u$ is an image of the human body, we can assume that it can be sparsely represented in a wavelet basis (cf.~\cite{candes2003curvelets,petersen2016shearlets}).
    The MRI reconstruction problem can hence be written as
    \begin{equation}
        \min_{v\in\mcal{M}_{\boldsymbol{1},k}} \norm{\hat{u} - \mcal{F}v},
    \end{equation}
    where the seminorm is chosen as $\abs{v}_{\omega} = \abs{v\pars{\omega}}$ and $\mcal{M}_{\boldsymbol{1},k}$ is defined with respect to the chosen wavelet basis.
    From Remark~\ref{rmk:residuals} we know that recovery requires \rip{\braces{\hat{u}} - \mcal{F}\mcal{M}_{\boldsymbol{1},k}}{\delta} and \rip{\braces{\hat{u} - \mcal{F}\apprum}}{\delta} the probabilities of which can be bounded by Corollary~\ref{cor:sample_size_rip}.
    
    In the following we only compute the variation constant since the Fourier transform is an isometry and does not change the covering number.
    Assuming $u\in\mcal{M}_{\boldsymbol{1},k}$ we can estimate
    \begin{equation}
        K\pars{U\pars{\hat{u}-\mcal{F}\mcal{M}_{\boldsymbol{1},k}}}
        \le K\pars{U\pars{\mcal{F}\mcal{M}_{\boldsymbol{1},2k}}} .
    \end{equation}
    To evaluate this, let $\psi$ be the mother wavelet and define the daughter wavelets $\psi_{a,b}\pars{t} = \frac{1}{\sqrt{a}}\psi\pars{\frac{t-b}{a}}$.
    Due to basic properties of the Fourier transform $\hat{\psi}_{a,b}\pars{\omega} = \sqrt{a}\hat{\psi}\pars{a\omega}\exp\pars{-ia\omega}$ and since the daughter wavelets are normalized we obtain
    \begin{equation}
        K\pars{U\pars{\inner{\hat{\psi}_{a,b}}}}
        = \frac{\norm{\hat{\psi}_{a,b}}_{L^\infty}^2}{\norm{\hat{\psi}_{a,b}}_{L^2}^2} 
        = a\norm{\hat{\psi}}_{L^\infty}^2 .
    \end{equation}
    Note that $\hat{\psi}$ is the mother wavelet and therefore $\norm{\hat{\psi}}_{L^\infty}^2$ is constant.
    It can be concluded that many samples are needed to recover larger scale coefficients but fewer samples for smaller scales.
    This suggests a multilevel approach where the small-scale coefficients are learned separately from the large-scale coefficients.
    This was already observed in the compressed sensing literature (cf.~\cite{adcock2017breaking_coherence_barrier}).
    Typically, these schemes use the classical unweighted notion of sparsity.
    For a recent application of weighted sparsity in the context of residual minimization in a sparse wavelet representation we refer to~\cite{daws2019weighted}.
    
    Due to the high variation constant of the large scale coefficients, it is sensible to incorporate as much information as possible into this model class.
    In the spirit of works like~\cite{chen2017manifoldConstrained_lowRank}, this can for example be achieved by means of manifold constraints.
    These manifolds can either be estimated for a single patient (cf.~\cite{meng2020manifoldConstrained_lowRank_MRI}) or for multiple patients when it can be assumed that the large-scale structures remain similar for different patients.
    In this way the image $u$ is decomposed (approximately) as a sum of a background image modelling the healthy tissue and a foreground image modelling the pathological lesion.
    
    Note that if the mother wavelet $\Psi$ is differentiable we can instead consider the semi-norm $\abs{v}_\omega := \sqrt{1+\omega^2}\abs{v\pars{\omega}}$, which corresponds to the $H^1$-norm in the physical domain.
    Computing the variation constant is however out of the scope of our discussion.
\end{example}
}

\section{Discussion}
\label{sec:discussion}

\revision{
The nonlinear least squares method is probably the easiest and currently the most commonly used setting in machine learning regression.
In Section~\ref{sec:main_results} we derive an error bound for the nonlinear least squares estimator~\eqref{eq:emp_min} that can be used with arbitrary model classes.
This result is based on a \emph{restricted isometry property (RIP)}, which we prove to hold with high probability when the number of samples is sufficiently large.

To put our theory into perspective, we apply it to well-known model classes and compare the results to the near optimal bounds that often already exist in the literature.
In the cases of linear spaces (Section~\ref{sec:linear}), functions with sparse representation (Section~\ref{sec:sparse}) and low-rank tensors (Section~\ref{sec:cp-tensors}), we obtain asymptotic bounds which differ from these near optimal ones by a polynomial factor.
This means that our analysis does not provide optimal complexity bounds when the number of samples should be determined a priori and when sampling is costly (i.e.\ when it is imperative to require as few samples as possible).
We however assume that a more meticulous application of modern concentration arguments (like~\cite{dirksen2015generic_chaining}) would close this gap.
We also obtain first bounds for the sample complexity for rank-$1$ measured low-rank tensors in Section~\ref{sec:cp-tensors}.
These bounds however only improve the sample complexity of full-rank tensors by a logarithmic term.
An intuition for this result is provided by matrix recovery where it is known that regularity in the form of incoherence is needed in addition to the low-rank property.
As a first remedy we suggest to impose additional regularity assumptions on the model class as was done in~\cite{goette2020local_interactions}. %
We however believe that this problem can be handled by taking the regularity of the function $u$ that we want to approximate into account.
Figure~\ref{fig:least_squares_regularity} illustrates this behaviour.
The model class used for all three experiments is the same and only the regularity of the function varies.
Even though the best approximation error in all three cases is bounded by $10^{-3}$, we can observe how the empirical approximations deteriorate with decreasing regularity. 
The relative errors for the empirical approximation increase from $10^{-2}$ to $10^{1}$.
This phenomenon will be investigated in future research.
\begin{figure}
     \centering
     \begin{subfigure}[b]{0.3\textwidth}
         \centering
         \includegraphics[width=\textwidth]{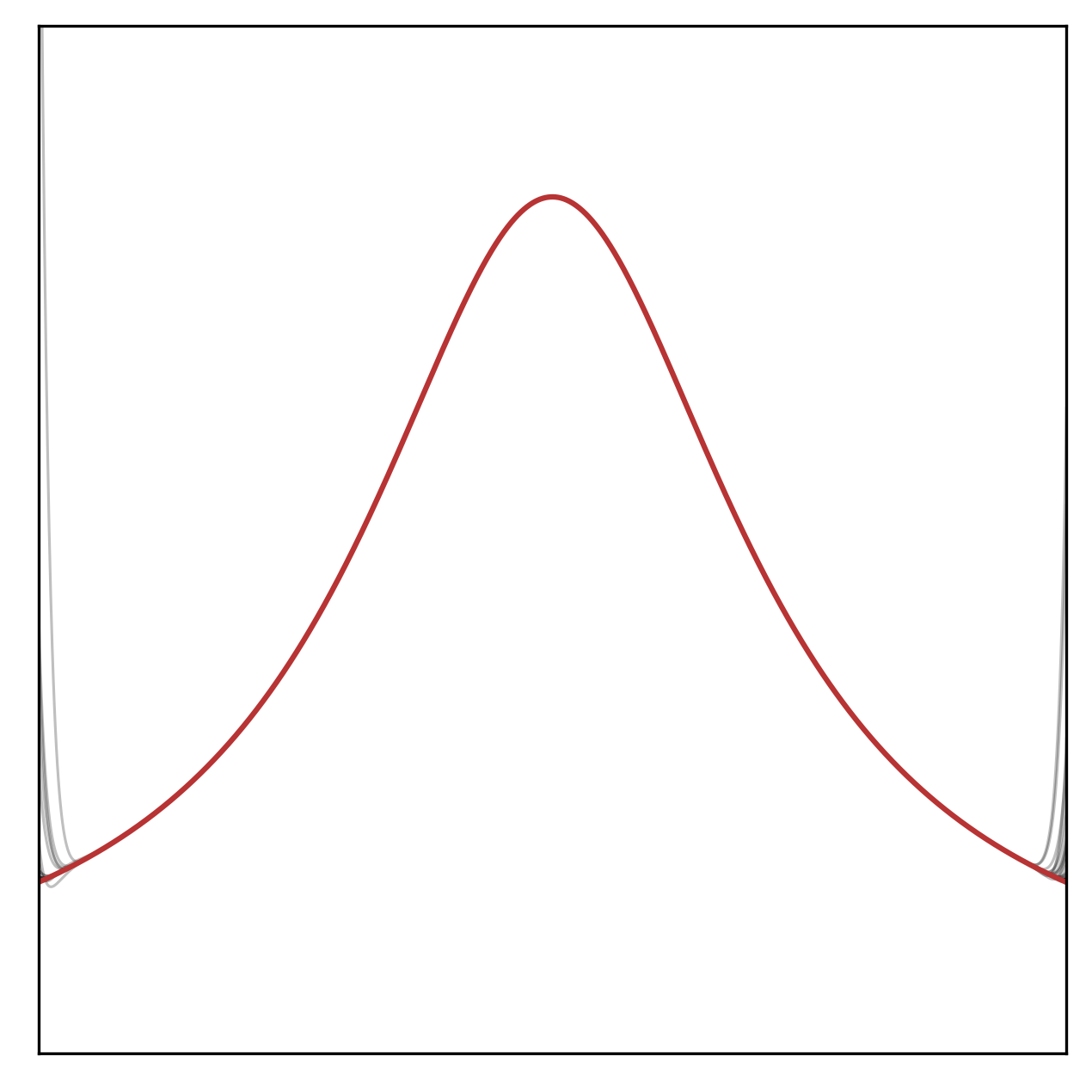}
         \caption{$f\pars{x} = \frac{1}{1+5x^2}$}
         \label{fig:c5}
     \end{subfigure}
     \hfill
     \begin{subfigure}[b]{0.3\textwidth}
         \centering
         \includegraphics[width=\textwidth]{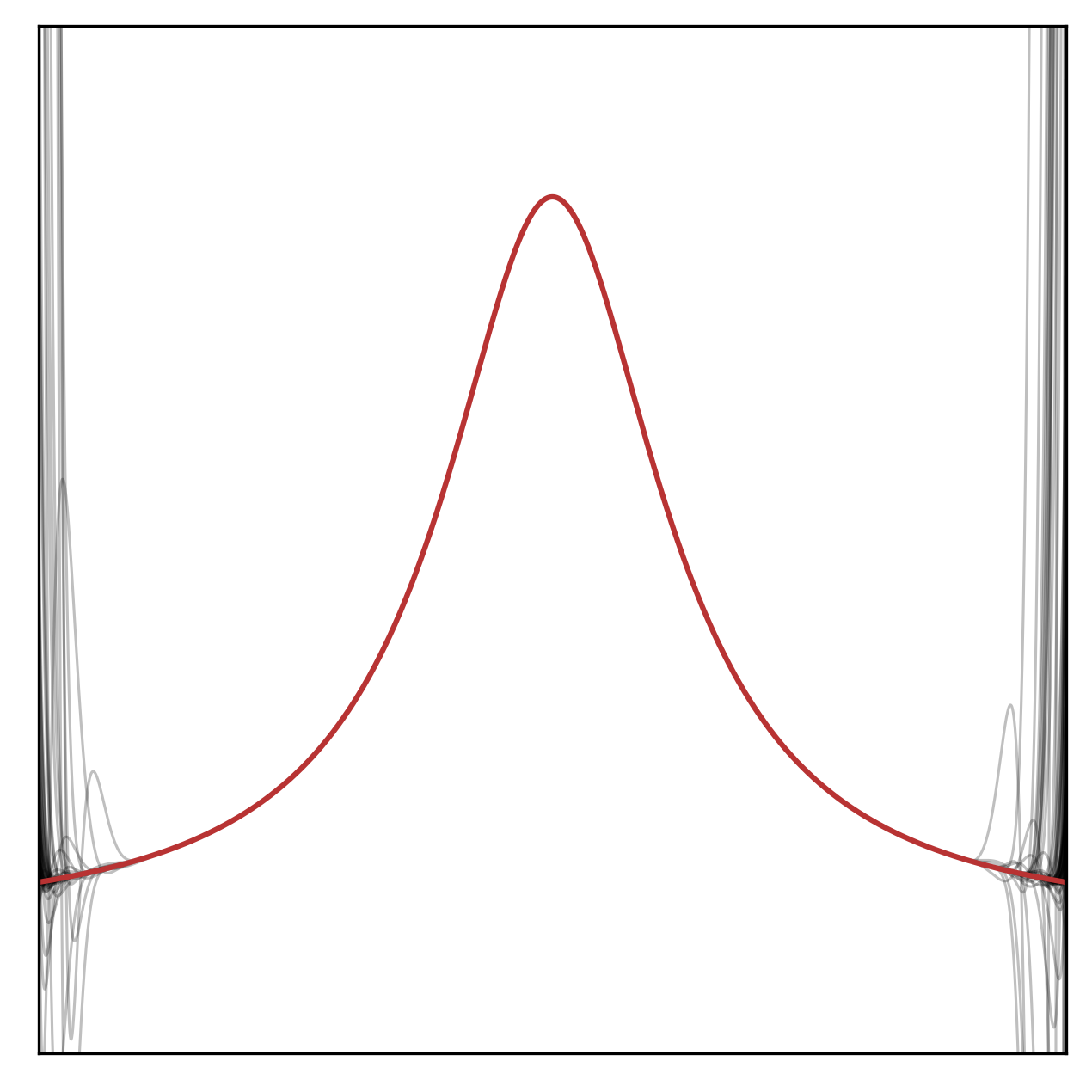}
         \caption{$f\pars{x} = \frac{1}{1+15x^2}$}
         \label{fig:c15}
     \end{subfigure}
     \hfill
     \begin{subfigure}[b]{0.3\textwidth}
         \centering
         \includegraphics[width=\textwidth]{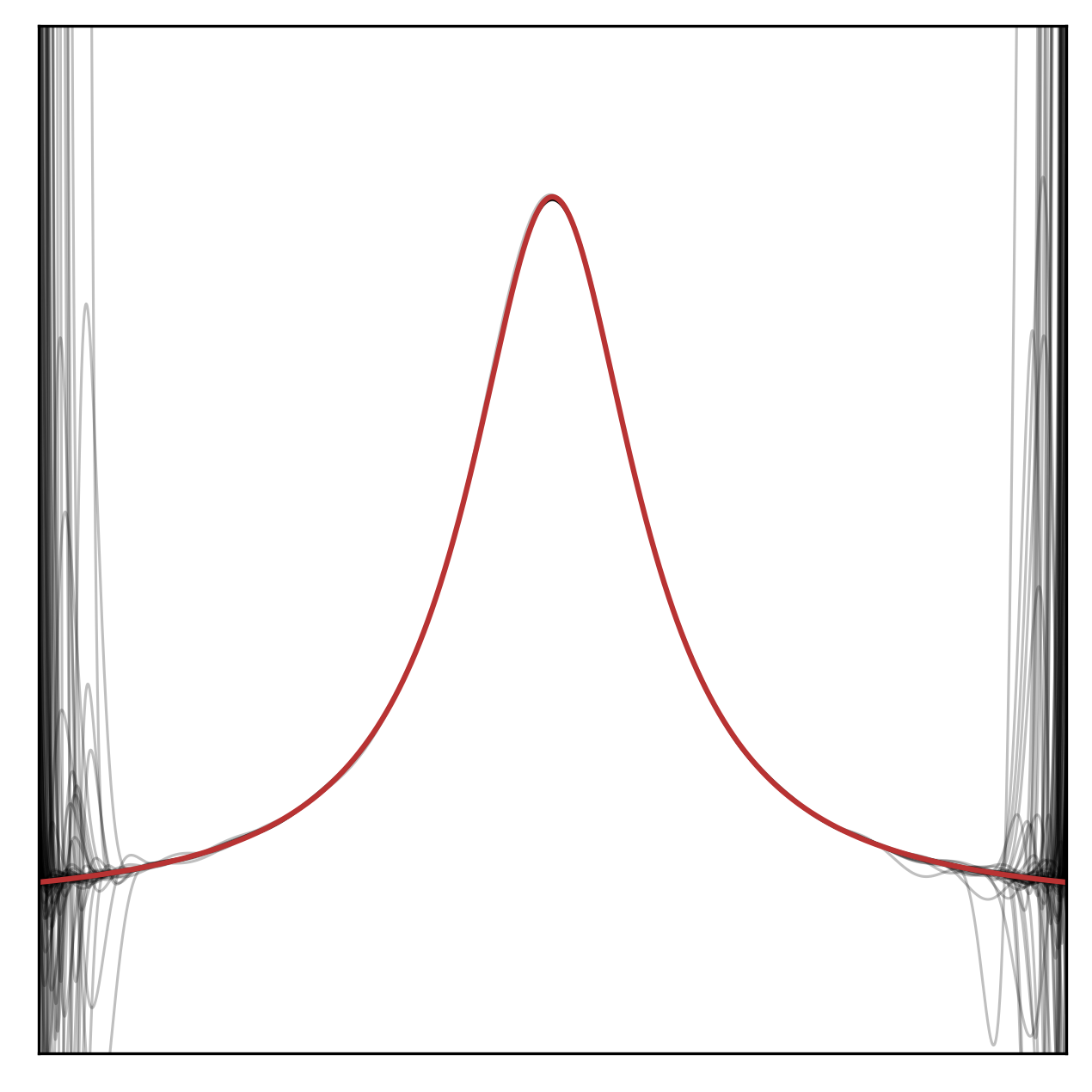}
         \caption{$f\pars{x} = \frac{1}{1+25x^2}$}
         \label{fig:c25}
     \end{subfigure}
    \caption{
    Overlaid least squares approximations of the function $f\pars{x} = \frac{1}{1+c x^2}$ (red) by Legendre polynomials of degree $29$.
    Different approximations correspond to different random draws of $n = 100$ sampling points from the uniform measure on $\bracs{-1,1}$.
    }
    \label{fig:least_squares_regularity}
\end{figure}

Despite the mentioned limitation, we nevertheless obtain qualitatively similar results to what is reported with more specialized approaches.
In particular this concerns the emergence of an optimal sampling measure in Section~\ref{sec:linear}, the importance of weighted sparsity in Section~\ref{sec:sparse} and the advantage of multilevel sampling in Example~\ref{ex:MRI}.
The generality of our theory also allows us to combine these result and derive an optimal weight function for weighted sparsity in Lemma~\ref{lem:weighted_sparsity_weight_function}.
Since these results rely only on an estimation of the RIP, they can be compared to results on weighted $\ell^1$-minimization.
We observe an improvement over the unweighted case.

In a final section, the dependence of the sample complexity on the seminorm that is used is investigated.
We observe faster convergence when stronger norms are used and provide a theoretical reasoning for this effect.

Despite several remaining problems, we hope that this work is a promising first step towards a general theory for the sample complexity of the nonlinear least squares problem.
We also want to emphasise that although our discussion is limited to well-known model classes, the developed theory can be applied to arbitrary model classes which may even be constructed empirically by methods such as manifold learning.
}

\section{Acknowledgements}

We thank the anonymous referees for suggestions that helped to significantly improve the manuscript and also to correct an error.
We also thank Leon Sallandt, Mathias Oster and Michael G\"otte for fruitful discussions.

M.~Eigel acknowledges support by the DFG SPP 1886.
R.~Schneider was supported by the Einstein Foundation Berlin.
P.~Trunschke acknowledges support by the Berlin International Graduate School in Model and Simulation based Research (BIMoS).

\printbibliography

\appendix

\section{Proof of Lemma~\ref{lem:generalization_error}}
\label{app:lem:generalization_error}

The proof consists of two steps.
In the first step we derive Lemma~\ref{lem:sup-to-max} to show that there exists $\nu\in\mbb{N}$ and $\braces{u_j}_{j\in\bracs{\nu}}\subseteq U\pars{A}$ such that
\begin{align}
    \mbb{P}\bracs*{\sup_{u\in U\pars{A}} \abs{\norm{u}^2 - \norm{u}^2_n} > \delta}
    &\le \mbb{P}\bracs*{\max_{1\le j\le\nu} \abs{\norm{u_j}^2 - \norm{u_j}^2_n} > \tfrac{\delta}{2}}.
\intertext{Using a union bound argument it follows that}
    \mbb{P}\bracs*{\max_{1\le j\le\nu} \abs{\norm{u_j}^2 - \norm{u_j}^2_n} > \tfrac{\delta}{2}} &\le \sum_{1\le j\le\nu} \mbb{P}\bracs*{\abs{\norm{u_j}^2 - \norm{u_j}^2_n} > \tfrac{\delta}{2}} \\
    &\le \nu \max_{1\le j\le\nu} \mbb{P}\bracs*{\abs{\norm{u_j}^2 - \norm{u_j}^2_n} > \tfrac{\delta}{2}}.
\end{align}
In the second step we prove Lemma~\ref{lem:applied-hoeffding} which allows us to bound the probability
\begin{align}
    \mbb{P}\bracs*{\abs{\norm{u_j}^2 - \norm{u_j}^2_n} > \tfrac{\delta}{2}}
    \le 2 \exp\pars{-\tfrac{\delta^2n}{2K^2}}
\end{align}
for each $1\le j \le \nu$ by a standard concentration inequality.
Combining both inequalities yields the statement.

In the following we are concerned with proving Lemmas~\ref{lem:sup-to-max} and \ref{lem:applied-hoeffding} which both rely on properties of the function $\ell_y : u \mapsto w\pars{y} \abs{u}_y^2$.

\begin{lemma}\label{lem:ell_properties}
    The function $\ell_y : u \mapsto w\pars{y} \abs{u}_y^2$ has the properties
    \begin{itemize}
        \item $\abs{\ell_y\pars{u}} \le K$ and
        \item $\abs{\ell_y\pars{u} - \ell_y\pars{v}} \le 2\sqrt{K}\norm{u-v}_{w,\infty}$
    \end{itemize}
    for all $u,v \in U\pars{A}$.
\end{lemma}
\renewcommand{\qedsymbol}{\ensuremath{\square}}
\begin{proof}
    Let $u,v\in U\pars{A}$.
    The first statement follows immediately by
    \begin{equation}
        \abs{\ell_y\pars{u}} \le \sup_{u\in U\pars{A}} \esssup_{y\in Y} w\pars{y}\abs{u}_y^2 = K .
    \end{equation}
    To prove the second statement we consider the seminorm $\mcal{k}_y := \sqrt{\ell_y}$ and
    use the reverse triangle inequality 
    \begin{equation}
        \abs{\mcal{k}_y\pars{u} - \mcal{k}_y\pars{v}} \le \mcal{k}_y\pars{u-v} \le \esssup_{y\in Y} \mcal{k}_y\pars{u-v} = \norm{u-v}_{w,\infty}.
    \end{equation}
    Since $\mcal{k}_y$ is bounded by $\sqrt{K}$, we can use the Lipschitz continuity of $x\mapsto x^2$ on $[-\sqrt{K}, \sqrt{K}]$ to conclude
    \begin{equation}
        \abs{\ell_y\pars{u} - \ell_y\pars{v}} \le 2\sqrt{K}\abs{\mcal{k}_y\pars{u} - \mcal{k}_y\pars{u}} \le 2\sqrt{K} \norm{u-v}_{w,\infty}. \tag*{\qedhere}
    \end{equation}
\end{proof}
\renewcommand{\qedsymbol}{\ensuremath{\blacksquare}}

As an intermediate step we first prove Lemma~\ref{lem:sup-union-bound} from which Lemma~\ref{lem:sup-to-max} follows almost immediately.

\begin{lemma}\label{lem:sup-union-bound}
    Let $\nu := \nu_{\norm{\bullet}_{w,\infty}} \pars*{U(A), \tfrac{\delta}{8\sqrt{K}}}$ and $\braces{u_j}_{j\in\bracs{\nu}}$ be the centres of the corresponding covering.
    Then almost surely
    \begin{equation}
        \sup_{u\in U\pars{A}} \abs{\norm{u}^2 - \norm{u}_n^2}
        \le \tfrac{\delta}{2} + \max_{1\le j\le\nu} \abs*{\norm{u_j}^2 - \norm{u_j}_{n}^2}.
    \end{equation}
\end{lemma}
\renewcommand{\qedsymbol}{\ensuremath{\square}}
\begin{proof}
    Let $u \in U\pars{A}$ be given.
    Then by definition of the $\braces{u_j}_{j\in\bracs{\nu}}$, there is a specific $u_j$ with $\norm{u - u_j}_{w,\infty} \le \frac{\delta}{8 \sqrt{K}}$.
    By Lemma~\ref{lem:ell_properties} and Jensen's inequality we know that
    \begin{align}
        \abs{\norm{u}^2 - \norm{u_j}^2}
        &\le \int_Y \abs*{\ell_y\pars{u} - \ell_y\pars{u_j}} \dmx{\rho}{y}
        \le 2\sqrt{K} \norm{u-v}_{w,\infty}
        \le \tfrac{\delta}{4} \\
        \intertext{and almost surely}
        \abs{\norm{u}_n^2 - \norm{u_j}_n^2}
        &\le \tfrac{1}{n} \sum_{i=1}^n \abs*{\ell_{y_i}\pars{u} - \ell_{y_i}\pars{u_j}}
        \le 2\sqrt{K} \norm{u-v}_{w,\infty}
        \le \tfrac{\delta}{4} .
    \end{align}
    Therefore, by triangle inequality,
    \begin{align}
        \abs{\norm{u}^2 - \norm{u}^2_n}
        &\le \abs{\norm{u}^2 - \norm{u}^2_n - \pars{\norm{u_j}^2 - \norm{u_j}^2_n}} + \abs{\norm{u_j}^2 - \norm{u_j}^2_n} \\
        &\le \abs{\norm{u}^2 - \norm{u_j}^2} + \abs{\norm{u}^2_n - \norm{u_j}^2_n} + \abs{\norm{u_j}^2 - \norm{u_j}^2_n} \\
        &\le \tfrac{\delta}{2} + \abs{\norm{u_j}^2 - \norm{u_j}^2_n} \quad\text{ almost surely.}
    \end{align}
    Taking the maximum concludes the proof.
\end{proof}
\renewcommand{\qedsymbol}{\ensuremath{\blacksquare}}

\begin{lemma}\label{lem:sup-to-max}
    Let $\nu := \nu_{\norm{\bullet}_{w,\infty}} \pars*{U(A), \tfrac{\delta}{8\sqrt{K}}}$ and $\braces{u_j}_{j\in\bracs{\nu}}$ be the centres of the corresponding covering.
    Then
    \begin{equation}
        \mbb{P}\bracs*{\sup_{u\in U\pars{A}} \abs{\norm{u}^2 - \norm{u}^2_n} > \delta}
        \le \mbb{P}\bracs*{\max_{1\le j\le\nu} \abs{\norm{u_j}^2 - \norm{u_j}^2_n} > \tfrac{\delta}{2}} .
    \end{equation}
\end{lemma}
\renewcommand{\qedsymbol}{\ensuremath{\square}}
\begin{proof}
    By Lemma~\ref{lem:sup-union-bound}
    \begin{equation}
        \sup_{u\in U\pars{A}} \abs{\norm{u}^2 - \norm{u}_n^2}
        \le \tfrac{\delta}{2} + \max_{1\le j\le\nu} \abs*{\norm{u_j}^2 - \norm{u_j}_{n}^2}
    \end{equation}
    holds almost surely. In this event we know that
    \begin{equation}
        \sup_{u\in U\pars{A}} \abs{\norm{u}^2 - \norm{u}^2_n} > \delta
        \Rightarrow \max_{1\le j\le\nu} \abs{\norm{u_j}^2 - \norm{u_j}^2_n} > \tfrac{\delta}{2}
    \end{equation}
    which concludes the proof.
\end{proof}
\renewcommand{\qedsymbol}{\ensuremath{\blacksquare}}

To prove Lemma~\ref{lem:applied-hoeffding} we first recall a standard concentration result from statistics.

\begin{lemma}[Hoeffding 1963]\label{lem:hoeffding}
    Let $\braces{X_i}_{i\in\bracs{N}}$ be a sequence of i.i.d.\ bounded random variables $\abs{X_i} \le M$ and define $\overline{X} \coloneqq \frac{1}{N}\sum_{i=1}^N X_i$. Then
    \begin{equation}
        \mbb{P}\bracs*{\abs{\mbb{E}\bracs{\overline{X}}-\overline{X}} \ge \delta} \le 2 \exp\pars*{-\tfrac{2\delta^2N}{M^2}} .
    \end{equation}
\end{lemma}

The proof of Lemma~\ref{lem:applied-hoeffding} is now a mere application of this result.

\begin{lemma}\label{lem:applied-hoeffding}
    Let $u_j \in U\pars{A}$ then
    \begin{align}
        \mbb{P}\bracs*{\abs{\norm{u_j}^2 - \norm{u_j}^2_n} > \tfrac{\delta}{2}}
        \le 2 \exp\pars{-\tfrac{n\delta^2}{2K^2}} .
    \end{align}
\end{lemma}
\renewcommand{\qedsymbol}{\ensuremath{\square}}
\begin{proof}[Proof of Lemma~\ref{lem:applied-hoeffding}]
    The statement follows from an application of Lemma~\ref{lem:hoeffding} to the sequence of random variables $\braces{\ell_{y_i}\pars{u_j}}_{i=1}^n$.
    Since the samples $y_i$ are i.i.d.\ the random variables $\ell_{y_i}\pars{u}$ are i.i.d.\ as well.
    Moreover, by Lemma~\ref{lem:ell_properties} the variables are bounded in absolute value by $K$.
    Therefore, the assumptions for Lemma~\ref{lem:hoeffding} are satisfied.
\end{proof}
\renewcommand{\qedsymbol}{\ensuremath{\blacksquare}}

\section{Proof of Theorem~\ref{thm:pointwise_K}} \label{app:thm:pointwise_K}

To prove the first assertion it suffices to show that $\hat b$ is measurable.
For this let $\braces{u_j}_{j=1}^\infty$ be a countable dense subset in $\mcal{M}$.
Then
\begin{equation}
    \hat b\pars{y} := \sup_{u\in\mcal{M}} \abs{u}_y^2 = \sup_{j\in\mbb{N}} \abs{u_j}_y^2
\end{equation}
is the supremum over a countable set of measurable functions and as such it is measurable.

\revision{If $A$ is $\norm{\bullet}$-bounded and $K\pars{A}$ is finite then $A$ is $\norm{\bullet}_{w,\infty}$-bounded.
From this we can conclude the integrability of $\hat{b}$ by
\begin{equation}
    \int_Y \hat{b}\pars{y} \dmx{\rho}{y} \le \sup_{y\in Y} w\pars{y} \sup_{v\in A} \abs{v}_y^2 \int_Y w\pars{y}^{-1} \dmx{\rho}{y} = \sup_{v\in A} \norm{v}_{w,\infty}^2 .
\end{equation}
It remains to show that the weight function $w = \norm{\hat{b}}_{L^1\pars{Y,\rho}}\hat{b}^{-1}$ is indeed optimal.}
We only sketch the proof of \revision{this} assertion.

By substituting $w = \pars{v\hat{b}}^{-1}$, the minimization problem
\begin{equation}
    \min_w K_{w} \quad\text{s.t.}\quad w \ge 0 \text{ and } \norm{w^{-1}}_{L^1\pars{Y,\rho}} = 1
\end{equation}
is equivalent to
\begin{equation}
    \min_v \norm{v^{-1}}_{L^\infty\pars{Y,\rho}} \quad\text{s.t.}\quad v > 0 \text{ and } \int_Y \hat b v \dx{\rho} = 1,
\end{equation}
which is a non-convex optimization problem under linear constraints.
The assertion is then equivalent to the statement that the minimal $v$ is a constant function \revision{and the constraint $\int_Y \hat b v \dx{\rho} = 1$ implies $w=\norm{\hat{b}}_{L^1\pars{Y,\rho}} \hat{b}^{-1}$.}

\revision{
To prove that a minimal $v$ has to be constant, let $\Omega_1\subseteq Y$ be any measurable subset and $\Omega_2 := Y\setminus\Omega_1$.
Then $v$ can be written as $v = \alpha_1 v_1 + \alpha_2 v_2$ with
\begin{equation}
    \alpha_{k} := \norm{v^{-1}}_{L^{\infty}\pars{\Omega_{k}, \rho}}^{-1} \qquad\text{and}\qquad v_{k} := \frac{v\chi_{\Omega_{k}}}{\alpha_{k}} \qquad\text{for } k=1,2 .
\end{equation}
Now observe that 
\begin{equation}
    \norm{v^{-1}}_{L^{\infty}\pars{Y,\rho}} = \norm{v^{-1}}_{L^{\infty}\pars{\Omega_{1}, \rho}} \vee \norm{v^{-1}}_{L^{\infty}\pars{\Omega_{2}, \rho}} = \alpha_1^{-1} \vee \alpha_2^{-1}.
\end{equation}
Moreover, $v>0$ implies $\alpha_{1}, \alpha_2 > 0$ and the linear constraint can hence be written as $\alpha_1 I_1 + \alpha_2 I_2 = 1$ with $I_{k} := \int_Y \hat{b}v_{k}\dx{\rho}$ for $k=1,2$.
Since $v$ is optimal, it must also satisfy
\begin{equation}
    \min_{\alpha_1,\alpha_2} \alpha_1^{-1}\vee\alpha_2^{-1} \quad\text{s.t.}\quad \alpha_1,\alpha_2 > 0 \text{ and } \alpha_1 I_1 + \alpha_2 I_2 = 1.
\end{equation}
Figure~\ref{fig:opt_w} illustrates why the solution must be $\alpha_1 = \alpha_2$.
This means that an optimal function $v$ has to satisfy $\norm{v^{-1}}_{L^{\infty}\pars{\Omega_{1}, \rho}} = \norm{v^{-1}}_{L^{\infty}\pars{\Omega_{2}, \rho}}$.
The claim now follows since the subset $\Omega_1$ was chosen arbitrarily.
}

\begin{figure}
    \centering
    \begin{subfigure}[b]{0.475\textwidth}
        \centering
        \includegraphics[width=\textwidth]{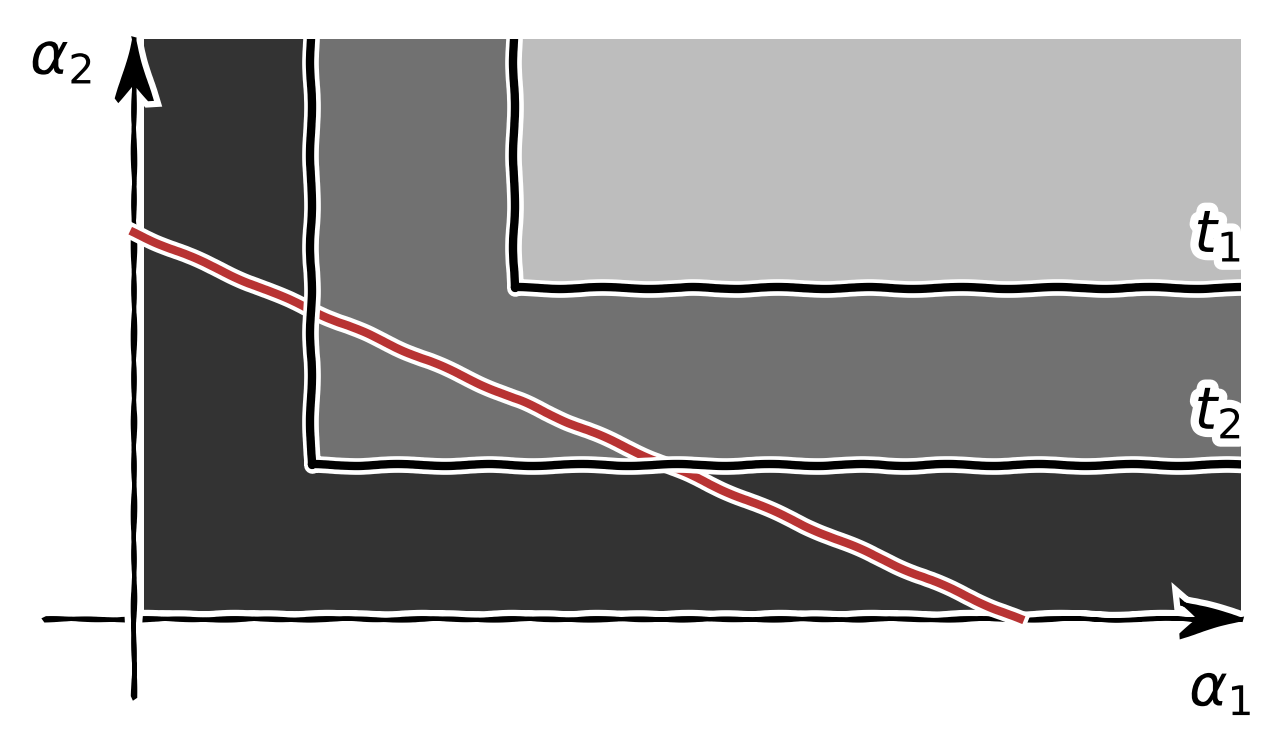}
    \end{subfigure}
    \hfill
    \begin{subfigure}[b]{0.475\textwidth}
        \centering
        \includegraphics[width=\textwidth]{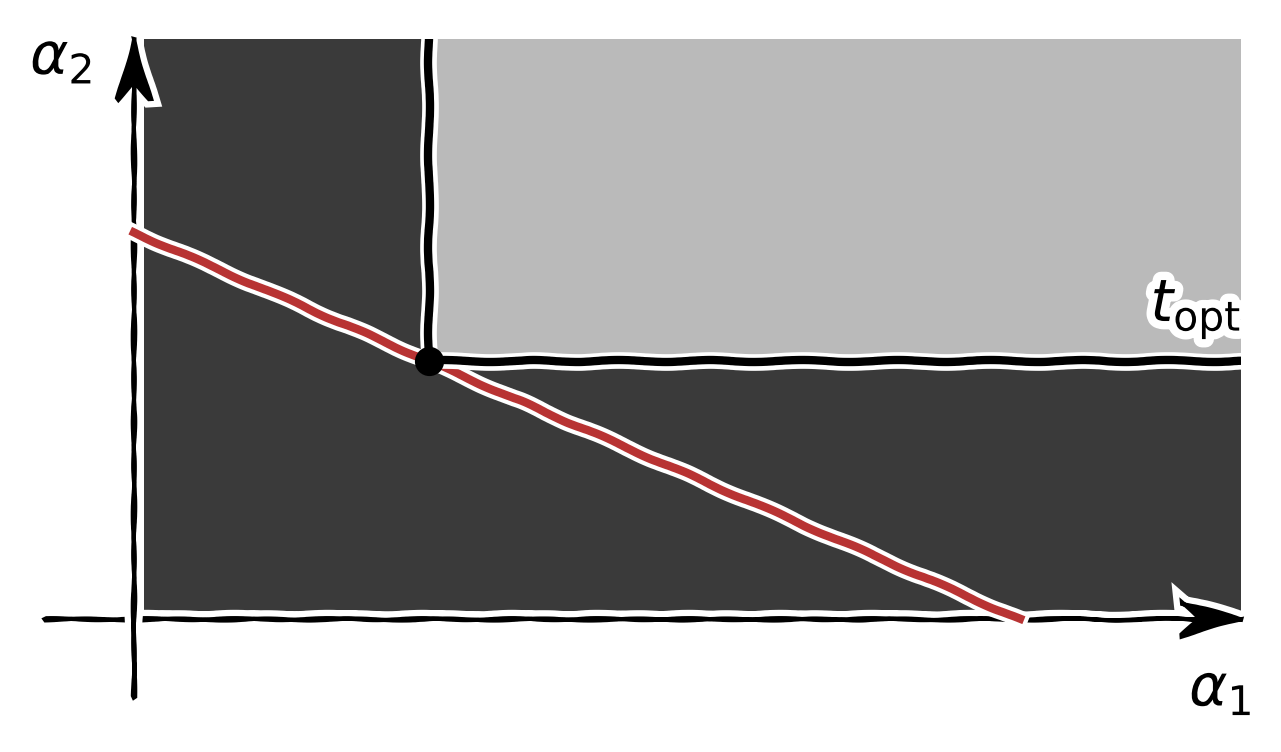}
    \end{subfigure}
    \caption{
        \revision{The set of feasible $\alpha_1, \alpha_2$ satisfying $\alpha_1,\alpha_2>0$ and $\alpha_1I_1 + \alpha_2I_2 = 1$ is displayed in red.
        Contour lines of the function $\pars{\alpha_1,\alpha_2}\mapsto \alpha_1^{-1}\vee\alpha_2^{-1}$ for $t_1 < t_2$ (left) and for the optimal value $t_{\textrm{opt}} = \alpha_1^{-1} = \alpha_2^{-1}$ (right) are drawn in black.}
    }
    \label{fig:opt_w}
\end{figure}

\section{Proof of Example \ref{sec:Hk_seminorm}}\label{app:sec:Hk_seminorm}

Recall that $\mcal{V} := H^m\pars{Y,\rho}$ where $Y\subseteq \mbb{R}^d$ is a Lipschitz domain and $A \subseteq \mcal{H} := H^{M}\pars{Y, \rho}$ with $\ell := M-m > \frac{d}{2}$.
It was shown in~\cite{extension_theorems_sobolev} that since $Y$ is Lipschitz $H^m\pars{Y}$ can be embedded isometrically into $H^m\pars{\mbb{R}^d}$.
This means that we can restrict our analysis to the case $Y = \mbb{R}^d$.
Since $\ell > \frac{d}{2}$, the Sobolev embedding theorem ensures that $D^{\alpha} v \in H^\ell\pars{Y,\rho} \subseteq C^0\pars{Y}$.
This means that the seminorm of $H^m\pars{Y,\rho}$ can be represented by 
\begin{equation}
    \abs{v}_y^2
    = \sum_{\abs{\alpha}\le m} \abs{\bracs{D^{\alpha} v}\pars{y}}^2
    = \sum_{\abs{\alpha}\le m} \abs{L^\alpha_y v}^2
\end{equation}
with the family of linear operators $L^\alpha_y : v \mapsto \bracs{D^\alpha v}\pars{y}$.
In the following we compute
\begin{equation}
    \kappa\pars{y}
    = \norm{L_y}_{\mcal{L}\pars{\mcal{H}, \mbb{R}^{\abs{\braces{\abs{\alpha}\le m}}}}}^2
    = \sum_{\abs{\alpha}\le m} \norm{L^\alpha_y}{\mcal{H}^*}^2 .
\end{equation}

As in~\cite{novak2018reproducing} the Riesz representative of $L^\alpha_y$
\begin{equation}
    K^\alpha_y\pars{x} := \int_{\mbb{R}^d} \frac{\prod_{j=1}^d\pars{2\pi\iu u_j}^{\alpha_j} \exp\pars{2\pi\iu\pars{x-y}\cdot u}}{\sum_{\abs{\beta}\le m+l} \prod_{j=1}^d\pars{2\pi u_j}^{2\beta_j}} \dx{u}
\end{equation}
can be obtained by using the Fourier transform and some standard properties.
Thus,
\begin{align}
    \norm{L^\alpha_y}_{\mcal{H}^*}^2
    &= \norm{K^\alpha_y}_{\mcal{H}}^2
    = \inner{K^\alpha_y, \overline{K^\alpha_y}}_{\mcal{H}}
    = \bracs{D^\alpha \overline{K^\alpha_y}}\pars{y} \\
    &= \int_{\mbb{R}^d} \frac{\prod_{j=1}^d\pars{2\pi u_j}^{2\alpha_j}}{\sum_{\abs{\beta}\le {m+l}} \prod_{j=1}^d\pars{2\pi u_j}^{2\beta_j}} \dx{u}.
\intertext{By the change of variables $t_j = 2\pi u_j$}
    \norm{L^\alpha_y}_{\mcal{H}^*}^2 &= \frac{1}{\pars{2\pi}^d} \int_{\mbb{R}^d} \frac{\prod_{j=1}^d t_j^{2 \alpha_j}}{\sum_{\abs{\beta}\le m+l} \prod_{j=1}^d t_j^{2\beta_j}} \dx{t} .
\end{align}

The multinomial theorem states that
\begin{equation}
    \pars{1+\norm{t}_2^2}^m = \sum_{\abs{\alpha} \le m} \binom{m}{\alpha} \prod_{j=1}^d t_j^{2\alpha_j}.
\end{equation}
As a consequence,
\begin{equation}
    \sum_{\abs{\alpha} \le m} \prod_{j=1}^d t_j^{2\alpha_j}
    \le \pars{1+\norm{t}_2^2}^m
    \le \Gamma\pars{m+1} \sum_{\abs{\alpha} \le m} \prod_{j=1}^d t_j^{2\alpha_j} .
\end{equation}

This leads to the estimate
\begin{align}
    \sum_{\abs{\alpha}\le m} \norm{L^\alpha_y}_{\mcal{H}^*}^2
    &= \frac{1}{\pars{2\pi}^d} \int_{\mbb{R}^d} \frac{\sum_{\abs{\alpha}\le m} \prod_{j=1}^d t_j^{2 \alpha_j}}{\sum_{\abs{\beta}\le m+\ell} \prod_{j=1}^d t_j^{2\beta_j}} \dx{t} \\
    &\le \frac{\Gamma\pars{m+\ell+1}}{\pars{2\pi}^d} \int_{\mbb{R}^d} \frac{\pars{1+\norm{t}_2^2}^m}{\pars{1+\norm{t}_2^2}^{m+\ell}} \dx{t} \\
    &= \frac{\Gamma\pars{m+\ell+1}}{\pars{2\pi}^d} \int_{\mbb{R}^d} \frac{\dx{t}}{\pars{1+\norm{t}_2^2}^{\ell}} \\
    &= \frac{\Gamma\pars{m+\ell+1}}{\pars{2\pi}^d} \frac{2\pi^{d/2}}{\Gamma\pars{\frac{d}{2}}} \int_{0}^{\infty} \frac{s^{d-1}}{\pars{1+s^2}^{\ell}} \dx{s}.
\end{align}

The recurrence relation (2.147) in~\cite{gradshteyn1943tableOI} together with $\ell>\frac{d}{2}$ yields
\begin{equation}
    \int_0^\infty \frac{s^{d-1}}{\pars{1+s^2}^l} \dx{s}
    = \frac{d-2}{2\ell - d} \int_0^\infty \frac{s^{d-3}}{\pars{1+s^2}^\ell} \dx{s}
    = \ldots
    = \frac{\Gamma\pars{\ell-\frac{d}{2}} \Gamma\pars{\frac{d}{2}}}{2\Gamma\pars{\ell}}.
\end{equation}
Consequently,
\begin{equation}
    \kappa\pars{y}
    \le \pars{2\sqrt{\pi}}^{-d} \frac{\Gamma\pars{m+\ell+1} \Gamma\pars{\ell-\frac{d}{2}}}{\Gamma\pars{\ell}} .
\end{equation}

\end{document}